\documentclass[12pt,reqno,centertags,draft]{amsart}
\usepackage{amsmath}
\usepackage{}
\usepackage{color}
\usepackage{amsfonts}
\usepackage{a4wide}
\usepackage{amsmath,amsthm,amssymb,amscd}
\usepackage{latexsym}

\allowdisplaybreaks
\numberwithin{equation}{section}

\newtheorem{theorem}{Theorem}[section]
\newtheorem{proposition}[theorem]{Proposition}

\newtheorem{lemma}[theorem]{Lemma}

\theoremstyle{definition}

\begin{document}
\title
[Positive solutions to multi-critical elliptic problems]
{Positive solutions to multi-critical elliptic problems}

\maketitle
\begin{center}
\author{Fanqing Liu, \ \  Jianfu Yang}
\footnote{ Email addresses: fanqliu@163.com, jfyang\_2000@yahoo.com, xiaohui.yu@szu.edu.cn}
\end{center}
\begin{center}

\address{School of Mathematics
and Statistics, Jiangxi Normal University, Nanchang,
Jiangxi 330022, P. R. China }

and

\author{ Xiaohui Yu}

\address{College of Mathematics and Statistics, Shenzhen University, Shenzhen, Guangdong 518060, P. R. China }
\end{center}

\begin{abstract}
In this paper, we investigate the existence of multiple solutions to the following multi-critical elliptic problem
\begin{equation}\label{eq:0.1}
\left\{\begin{aligned}
-\Delta u & =\lambda |u|^{p-2}u +\sum_{i=1}^k(|x|^{-(N-\alpha_i)}*|u|^{2^*_i})|u|^{2^*_i-2}u\quad {\rm in}\quad \Omega,\\
&u\in H^1_0(\Omega)\\
\end{aligned}\right.
\end{equation}
in connection with the topology of the bounded domain $\Omega\subset \mathbb{R}^N, \,N\geq 4$, where $\lambda>0$,  $2^*_i=\frac{N+\alpha_i}{N-2}$ with $N-4<\alpha_i<N,\ \ i=1,2,\cdot\cdot\cdot, k$ are critical Hardy-Littlewood-Sobolev exponents and $2<p<22^*_{min}$ with $2^*_{min}=\min\{2^*_i, \ i=1,2,\cdot\cdot\cdot, k\}$. We show that there is $\lambda^*>0$ such that if $0<\lambda<\lambda^*$
problem \eqref{eq:0.1} possesses at least $cat_\Omega(\Omega)$ positive solutions. We also study the existence and uniqueness of solutions for the limit problem of \eqref{eq:0.1}.

 {\bf Key words }:  Multi-critical problem, Multiple solutions, elliptic equation.

{\bf MSC2020:} 35J10, 35J20, 35J61

\end{abstract}

\bigskip

\section{Introduction}

\bigskip

In  this paper, we investigate the existence of multiple positive solutions to the following multi-critical elliptic problem
\begin{equation}\label{eq:1.1}
\left\{\begin{aligned}
-\Delta u & =\lambda |u|^{p-2}u +\sum_{i=1}^k(|x|^{-(N-\alpha_i)}*|u|^{2^*_i})|u|^{2^*_i-2}u\quad {\rm in}\quad \Omega,\\
&u\in H^1_0(\Omega)\\
\end{aligned}\right.
\end{equation}
in connection with the topology of the bounded domain $\Omega\subset \mathbb{R}^N, \,N\geq 4$, where $\lambda>0$, $H^1_0(\Omega)$ is the usual Sobolev space equipped with the norm $\|u\|_{H^1_0(\Omega)} = (\int_\Omega|\nabla u|^2\,dx)^{\frac 12}$, $2^*_i=\frac{N+\alpha_i}{N-2}$ with $N-4<\alpha_i<N,\ \ i=1,2,\cdot\cdot\cdot, k$ are critical Hardy-Littlewood-Sobolev exponents and $2<p<22^*_{min}$ with $2^*_{min}=\min\{2^*_i, \ i=1,2,\cdot\cdot\cdot, k\}$.

If $k = 1$ and $p=2$, equation  \eqref{eq:1.1} becomes the critical Choquard problem
\begin{equation}\label{eq:1.2}
\left\{\begin{aligned}
-\Delta u  &=\lambda u + (|x|^{-(N-\alpha)}*|u|^{2^*_\alpha})u^{2^*_\alpha-1}\quad  {\rm in}\quad \Omega,\\
&u\in H^1_0(\Omega),\\
\end{aligned}\right.
\end{equation}
where $2^*_\alpha=\frac{N+\alpha}{N-\alpha}$.  Such a problem  has been extensively studied for subcritical and critical cases in \cite{GY, GhP, Go, GS, MS, MS1} and references therein.
Especially, a solution of critical problem \eqref{eq:1.2} was found in \cite{GY} by the mountain pass lemma in spirit of \cite{BN}. In \cite{GY},
the compactness of the $(PS)_c$   sequence associated with problem \eqref{eq:1.2} is retained if $c$ is strictly less than a threshold
value related to the best constant
\begin{equation}\label{eq:1.3}
S_{\alpha} = \inf_{u\in D^{1,2}(\mathbb{R}^N)\setminus\{0\}}\frac{\int_{\mathbb{R}^N}|\nabla u|^2\,dx}{\bigg(\int_{\mathbb{R}^N}\int_{\mathbb{R}^N}\frac{|u(x)|^{2^*_\alpha}|u(y)|^{2^*_\alpha}}{|x-y|^{N-\alpha}}\,dxdy\bigg)^{\frac{N-2}{N+\alpha}}}
\end{equation}
of the Hardy-Littlewood-Sobolev inequality. For multi-critical problems, such a threshold value varies from problem to problem. Indeed, for the doubly critical problem defined in a bounded domain $\Omega\subset\mathbb{R}^N$
\begin{equation}\label{eq:1.4}
\left\{\begin{aligned}
-\Delta u  &=\mu \frac{u}{|x|^2} + K(x)\frac{u^{2^*(s)-1}}{|x|^s} +Q(x)\frac{u^{2^*(t)-1}}{|x-x_0|^t} +f(u)\quad  {\rm in}\quad \Omega,\\
&u>0\ \ \quad {\rm in} \quad\Omega,\\
&u=0\ \ \quad {\rm on} \quad \partial\Omega,\\
\end{aligned}\right.
\end{equation}
where $x_0\in\Omega$ and $x_0\not= 0$, $2^*(s)=\frac{2(N-s)}{N-2}$ and $2^*(t)=\frac{2(N-t)}{N-2}$ are the critical Sobolev-Hardy exponents, it was proved in \cite{GP} that the
$(PS)_c$ condition holds true if
$$
c\in (0,\min\{\frac{2-s}{2(N-s)}K(0)^{\frac{N-2}{s-2}}S_{\mu, s}^{\frac{N-s}{2-s}}, \frac{2-t}{2(N-t)}Q(x_0)^{\frac{N-2}{t-2}}S_{0,t}^{\frac{N-t}{2-t}}\}),
$$
where $S_{\mu, s}$ and $S_{0,t}$ are the best Sobolev-Hardy constants defined in \cite{GP}. In general, it was studied in \cite{KL} the existence of solutions for  multi-critical  Sobolev-Hardy problems with different singularities. The situation becomes different if $x_0=0$, or all Sobolve-Hardy terms have the same singular point. In this case, the multi-critical problem
\begin{equation}\label{eq:1.5}
\left\{\begin{aligned}
-\Delta u  &=\frac {\lambda}{|x|^s}u^{p-1} +\sum_{i=1}^l \frac {\lambda_i}{|x|^{s_i}}u^{2^*(s_i)-1} + u^{2^*-1}\quad  {\rm in}\quad \Omega,\\
&u>0\ \ \quad {\rm in} \quad\Omega,\\
&u=0\ \ \quad {\rm on} \quad \partial\Omega,\\
\end{aligned}\right.
\end{equation}
was considered in \cite{G}. The threshold value for compactness relies on the ground state solution of the limit problem
\begin{equation}\label{eq:1.6}
-\Delta u  =\sum_{i=1}^l \frac {\lambda_i}{|x|^{s_i}}|u|^{2^*(s_i)-2}u + |u|^{2^*-2}u\quad  {\rm in}\quad \mathbb{R}^N,
\end{equation}
that is, a minimizer of the problem
\begin{equation}\label{eq:1.7}
P_{min}= \inf\{E(u): u\in \mathcal{M}_E\},
\end{equation}
where
\[
E(u) = \frac 12\int_{\mathbb{R}^N}|\nabla u|^2\,dx-\sum_{i=0}^l\frac {\lambda_i}{2^*(s_i)}\int_{\mathbb{R}^N}\frac{|u|^{2^*(s_i)}}{|x|^{s_i}}\,dx-\frac 1{2^*}\int_{\mathbb{R}^N}|u|^{2^*}\,dx
\]
and
\[
\mathcal{M}_E=\{u\in D^{1,2}(\mathbb{R}^N)\setminus\{0\}:\langle E'(u),u\rangle =0\}
\]
as well as
$$
D^{1,2}(\mathbb{R}^N)=\{u\in L^{\frac {2N}{N-2}}(\mathbb{R}^N): |\nabla u|\in L^2(\mathbb{R}^N)\}
$$
with the norm $\|u\|_{D^{1,2}(\mathbb{R}^N)}=(\int_{\mathbb{R}^N}|\nabla u|^2\,dx)^{\frac 12}$.
It is known from \cite{ZZ} that problem \eqref{eq:1.6} has a positive ground state solution $U$ satisfying
\begin{equation}\label{eq:1.8}
U(x)\leq C(1+|x|^{2-N})\quad |\nabla U(x)|\leq C(1+|x|^{1-N}).
\end{equation}
The decaying law in \eqref{eq:1.8} allows one to verify $(PS)_c$ condition for the functional associated with problem \eqref{eq:1.6} if $c\in (0,P_{min})$. An existence result of problem \eqref{eq:1.6} is then given in \cite{G}. We remark that similar problems with the singular point $0\in\partial\Omega$ have been extensively studied in \cite{CZZ, HLW, LL} etc, see also references therein.

In this paper, we first consider the existence of positive solution of problem \eqref{eq:1.1}. Such a solution will be found as a critical point of the functional
\begin{equation}\label{eq:1.9}
I_{\lambda,\Omega}(u) = \frac 12\int_{\Omega}|\nabla u|^2\,dx-\frac 1p\lambda\int_{\Omega} u_+^p\,dx -\sum_{i=1}^k\frac {1}{22^*_i}\int_{\Omega}\int_{\Omega}\frac{u_+(x)^{2^*_i}u_+(y)^{2^*_i}}{|x-y|^{N-\alpha_i}}\,dxdy
\end{equation}
in $H^1_0(\Omega)$, where $u_+=\max\{u,0\}$. In the context, for the simplicity we write $u^q$ as $u_+^q$ for $2\leq q\leq 2^*$ etc in functionals. It is standard to show that critical points of $I_{\lambda,\Omega}(u)$ are positive solutions of problem \eqref{eq:1.1}, see for instance \cite{W}. In order to verify $(PS)_c$  condition for  $I_{\lambda,\Omega}(u)$, that is,  any sequence $\{u_n\}\subset H^1_0(\Omega)$ such that $I_{\lambda,\Omega}(u_n)\to c,\ \ I'_{\lambda,\Omega}(u_n)\to 0$ as $n\to\infty$, contains a convergent subsequence, we need to take into account the limit problem of problem \eqref{eq:1.1}:
\begin{equation}\label{eq:1.10}
-\Delta u  = \sum_{i=1}^k(|x|^{-(N-\alpha_i)}*|u|^{2^*_i})|u|^{2^*_i-2}u\quad {\rm in}\quad \mathbb{R}^N.
\end{equation}
The threshold value for the compactness of $I_{\lambda,\Omega}(u)$ will be given by
\begin{equation}\label{eq:1.11}
m(\mathbb{R}^N) = \inf \{J(u): u\in \mathcal{M}_{\mathbb{R}^N}\},
\end{equation}
where
\begin{equation}\label{eq:1.12}
J(u) = \frac 12\int_{\mathbb{R}^N}|\nabla u|^2\,dx-\sum_{i=1}^k\frac {1}{22^*_i}\int_{\mathbb{R}^N}\int_{\mathbb{R}^N}\frac{u(x)^{2^*_i}u(y)^{2^*_i}}{|x-y|^{N-\alpha_i}}\,dxdy
\end{equation}
is the corresponding functional of problem \eqref{eq:1.10} and
\begin{equation}\label{eq:1.13}
\mathcal{M}_{\mathbb{R}^N} =\{u\in D^{1,2}(\mathbb{R}^N)\setminus\{0\}:\langle J'(u),u\rangle =0\}
\end{equation}
is the associated Nehari manifold. We will show that $I_{\lambda,\Omega}(u)$ satisfies $(PS)_c$ condition if $c\in (0, m(\mathbb{R}^N))$. To verify this condition, one has to know not only the existence of a ground state solution of problem \eqref{eq:1.10}, but also the decaying law of the solution at infinity. Our first result consists of the existence of ground state solutions of problem \eqref{eq:1.10} and their decay at infinity. Actually, we find a ground state solution of problem \eqref{eq:1.10} with an explicit form. The result is stated as follows.
\begin{theorem}\label{thm:1} Problem \eqref{eq:1.10} has a unique positive solution in the form
\[
U(x) = C\Big(\frac\varepsilon{\varepsilon^2+|x-x_0|^2}\Big)^{\frac {N-2}2}
\]
for $x_0\in\mathbb{R}^N$ and $\varepsilon>0$.
\end{theorem}

Our proof of the existence part in Theorem \ref{thm:1} is  elementary and straightforward in contrast with variational approaches used in \cite{YW} and \cite{ZZ} etc. The approach for the uniqueness
is based on the moving spheres method. Similar result for double critical problems was presented in \cite{YZZ}.

Using Theorem \ref{thm:1}, we prove that there is a positive solution for problem \eqref{eq:1.1}.
\begin{theorem}\label{thm:2} Suppose $N\geq 4$ and $2<p<22^*_{min}$, then $m_{\lambda,\Omega}$ is achieved and  problem \eqref{eq:1.1} possesses at least a positive solution.

\end{theorem}

Next, we consider the existence of multiple solutions of problem \eqref{eq:1.1} linking with the topology of the domain $\Omega$.
The study of the existence of positive solutions in non-contractible was initiated by Kazdan and Warner \cite{KW} and  Coron  \cite{C}.
Actually, the existence and the multiplicity of solutions of
\begin{equation}\label{eq:1.14}
\left\{\begin{aligned}
-\Delta u +&\lambda u=u^{q-1}\quad {\rm in}\quad \Omega, \\
&u>0\quad {\rm in}\quad \Omega, \\
&u\in H^1_0(\Omega)\\
\end{aligned}\right.
\end{equation}
linking to the topology of domain $\Omega$ have been obtained in \cite{C} for $\lambda=0$ and $q=2^*=\frac{2N}{N-2},\, N\geq 3$ provided $\Omega$ is non-contractible. Later on,
it was proved in \cite{BaC} that critical problem \eqref{eq:1.14} has a positive solution if $H_d(\Omega, \mathbb{Z}_2)\not=0$ for some $d$. In critical and near critical cases,  it is shown in \cite{AD, BC, BC1,R, W} that problem \eqref{eq:1.14} possesses at least $cat_\Omega(\Omega)$ positive solutions. The subcritical problem of \eqref{eq:1.14} was studied in  \cite{CP,CP1} in  exterior domains,  $cat_\Omega(\Omega)$ number positive high energy solutions was found by the Ljusternik-Schnirelman theory.

Multiple positive solutions for the nonlinear Choquard problem
\begin{equation}\label{eq:1.15}
\left\{\begin{aligned}
-\Delta u +\lambda u& =(|x|^{-(N-\alpha)}*|u|^q)|u|^{q-2}u\quad {\rm in}\quad \Omega, \\
&u\in H^1_0(\Omega)\\
\end{aligned}\right.
\end{equation}
are obtained in \cite{GhP} related to the Ljusternik-Schnirelman category $cat_\Omega(\Omega)$ of domain $\Omega$. It was proved that if  $\Omega$ is bounded and $q$ is closed to the critical exponent $2^*_\alpha = \frac{N+\alpha}{N-2}$, problem \eqref{eq:1.15} possesses at least $cat_\Omega(\Omega)$ positive solutions; similar results were obtained in \cite{Go}  for the critical case $q= 2^*_\alpha$. Moreover, positive high energy solutions
was also found in \cite{GS} for  the critical case in an annular domain. It seems that for the multi-critical problem \eqref{eq:1.1}, no multiple solutions linking to the topology of the domain $\Omega$ might be found in literatures. Our  result in Theorem \ref{thm:1} allows us to find such a result as follows.

\begin{theorem}\label{thm:3}Suppose $N\geq 4$ and $2<p<22^*_{min}$, then there exists $\lambda^*>0$ such that for $0<\lambda<\lambda^*$, problem \eqref{eq:1.1} possesses at least $cat_\Omega(\Omega)$ positive solutions.

\end{theorem}

We remark that one might easily imagine similar multiple results hold for problem \eqref{eq:1.6}, but it will encounter essential difficulties in doing so.

In the proof of Theorem \ref{thm:3}, a key ingredient is to show that the barycenter of functions in certain level set belonging to a neighborhood of $\Omega$. In general, this is done by the concentration-compactness principle. Such an argument seems hard to apply to multi-critical problems. We use improved Sobolev inequalities related to the Morrey space to replace the concentration-compactness principle, we prove essentially that any minimizing sequence of $m(\mathbb{R}^N)$ has a subsequence strongly converging in $D^{1,2}(\mathbb{R}^N)$ up to translations and dilations.

\bigskip

This paper is organized as follows. We first study the limit problem in Sections 2 and 3, then the existence of a positive solution is proved in Section 4. Multiple results are established in Section 5.

\bigskip

 \section{The Limit problem: existence}

\bigskip

In this section, we will show that the limit problem \eqref{eq:1.10} has a ground state solution. By a ground state of  equation \eqref{eq:1.10} we mean a solution that
minimizes the associated  functional $I(u)$ defined in \eqref{eq:1.12} among all nontrivial solutions. Furthermore, we prove the uniqueness of positive solutions for
equation \eqref{eq:1.10} up to translations and dilations, and work out explicitly form of the solution.

\bigskip

\begin{lemma}\label{lem:2.1}  The function
\begin{equation} \label{eq:2.6}
f(t) = \frac 12 t-\sum_{i=1}^k\frac 1{22^*_i}\bigg(\frac t{S_i}\bigg)^{2^*_i}
\end{equation}
has a unique maximum point, where $S_i =S_{\alpha_i}$ defined in \eqref{eq:1.3}.
\end{lemma}

\begin{proof} Since $f(t)>0$ if $t>0$ small and $f(t)\to-\infty$ if $t\to+\infty$, there exists $t_0\in (0, +\infty)$ such that
\[
f(t_0) =\max_{t\geq 0}f(t)
\]
and
\begin{equation} \label{eq:2.7}
1 = \sum_{i=1}^k S_i^{-2^*_i}t_0^{2^*_i-1}.
\end{equation}

Now we show the maximum point is unique. Suppose on the contrary, there would exist another maximum point $t_1\in (0, +\infty)$ of $f(t)$ satisfying
\begin{equation} \label{eq:2.8}
1 = \sum_{i=1}^k S_i^{-2^*_i}t_1^{2^*_i-1}.
\end{equation}
If $t_1\not = t_0$, we find a contradiction from \eqref{eq:2.7} and \eqref{eq:2.8}. The proof is complete.

\end{proof}

\bigskip
We consider the minimization problem $m(\mathbb{R}^N)$ defined in \eqref{eq:1.11}. We may prove as Lemma \ref{lem:2.1} that for any $u\in D^{1,2}(\mathbb{R}^N)$, there is a unique $t_u>0$ such that $t_u u\in \mathcal{M}_{\mathbb{R}^N}$. This enables us to show as \cite{W}
that
\begin{equation} \label{eq:2.5}
m(\mathbb{R}^N)= \inf_{u\not = 0}\max_{t>0}J(tu)=\inf_{\gamma\in\Gamma}\max_{t\in[0,1]}J(\gamma(t)),
\end{equation}
where
\[
\Gamma:=\{\gamma\in C([0,1], D^{1,2}(\mathbb{R}^N)): \gamma(0) =0, J(\gamma(1)) \leq 0, \gamma(1)\not = 0\}.
\]
Moreover, we may verify that there exists $c>0$ such that for every $u\in \mathcal{M}_{\mathbb{R}^N}$,  there holds $c\leq\|u\|_{ D^{1,2}(\mathbb{R}^N)}$, and  $m(\mathbb{R}^N)>0$.

It is known from \cite{GY} that $S_{\alpha}$ defined in \eqref{eq:1.3} is achieved if and only if
\begin{equation} \label{eq:2.9}
u_{\lambda,\xi}(x) = C\bigg(\frac \lambda {\lambda^2+|x-\xi|^2}\bigg)^{\frac {N-2}2},
\end{equation}
where $\lambda\in(0,+\infty)$ and $\xi\in \mathbb{R}^N$.   Based on Lemma \ref{lem:2.1} and \eqref{eq:2.9}, we have the following existence result.
\begin{proposition}\label{prop:2.1} There is a positive constant $C_0>0$ such that the minimization problem $m(\mathbb{R}^N)$ defined in \eqref{eq:1.11} is achieved by the function
\begin{equation} \label{eq:2.10}
V_{\varepsilon,\xi}(x) = C_0\bigg(\frac \varepsilon {\varepsilon^2+|x-\xi|^2}\bigg)^{\frac {N-2}2}
\end{equation}
for $\varepsilon>0$. Alternatively, $V_{\varepsilon,\xi}$ is a ground state solution  of problem \eqref{eq:1.10}.
\end{proposition}

\begin{proof} Denote $U(x) = u_{1,0}$, where $u_{1,0}$ is defined in \eqref{eq:2.9}. Since $I(u)$ is invariant under translations and dilations, we  need only to find a minimizer of $m(\mathbb{R}^N)$ in the form of $U(x)$. Observing that
\[
m(\mathbb{R}^N) = \inf_{u\not = 0}\max_{t>0}J(tu) = \inf_{\|u\|_{D^{1,2}(\mathbb{R}^N)} = 1}\max_{t>0}J(\sqrt{t}u)
\]
for $u\in D^{1,2}(\mathbb{R}^N)$, we deduce from the Hardy-Littlewood-Sobolev inequality that
\[
J(\sqrt{t}u)\geq \frac 12 t\int_{\mathbb{R}^N}|\nabla u|^2\,dx - \sum_{i=1}^k\frac {t^{2^*_i}}{22^*_i}\bigg(\frac 1 {S_i}\int_{\mathbb{R}^N}|\nabla u|^2\,dx\bigg)^{2^*_i}=f(t).
\]
Thus,
\[
\max_{t>0}J(\sqrt{t}u)\geq \max_{t>0}f(t)=f(t_0)
\]
implying
\[
m(\mathbb{R}^N)\geq f(t_0).
\]

On the other hand, we may choose $\tilde C>0$ such that $V=\tilde CU$ satisfying $\|V\|_{D^{1,2}(\mathbb{R}^N)} = 1$.
Noting that $S_i=S_{\alpha_i}$ is achieved by $U$, we have that
\begin{equation} \label{eq:2.11}
\begin{split}
J(\sqrt{t}V)&= \frac 12 t\int_{\mathbb{R}^N}|\nabla V|^2\,dx - \sum_{i=1}^k\frac {t^{2^*_i}\tilde C^{22^*_i}}{22^*_i}\bigg(\frac 1 {S_i}\int_{\mathbb{R}^N}|\nabla U|^2\,dx\bigg)^{2^*_i}\\
&= \frac 12 t\int_{\mathbb{R}^N}|\nabla V|^2\,dx - \sum_{i=1}^k\frac {t^{2^*_i}}{22^*_i}\bigg(\frac 1 {S_i}\int_{\mathbb{R}^N}|\nabla V|^2\,dx\bigg)^{2^*_i}\\
&=f(t).\\
\end{split}
\end{equation}
Hence,
\[
\max_{t>0}J(\sqrt{t}V) =f(t_0),
\]
which yields
\[
m(\mathbb{R}^N) = \inf_{\|u\|_{D^{1,2}(\mathbb{R}^N)} = 1}\max_{t>0}J(\sqrt{t}u)\leq f(t_0).
\]
Consequently,
\[
m(\mathbb{R}^N) =f(t_0) = J(\sqrt{t_0}V).
\]
Let $V_0(x)=\sqrt{t_0}V(x)=\sqrt{t_0}\tilde C U(x)$. By the fact
\[
0=f'(t_0)=\frac 12 -\sum_{i=1}^k\frac 1{2} {S_i}^{-2^*_i}t_0^{2^*_i-1},
\]
and $\|V\|_{D^{1,2}(\mathbb{R}^N)} = 1$, we obtain as \eqref{eq:2.11} that
\[\begin{split}
0 &=\int_{\mathbb{R}^N}t_0|\nabla V|^2\,dx -\sum_{i=1}^k  t_0^{2^*_i}\bigg({\frac 1{S_i}\int_{\mathbb{R}^N}|\nabla V|^2\,dx}\bigg)^{2^*_i}\\
&=\int_{\mathbb{R}^N}|\nabla V_0|^2\,dx -\sum_{i=1}^k  t_0^{2^*_i}\tilde C^{22^*_i}\bigg({\frac 1{S_i}\int_{\mathbb{R}^N}|\nabla U|^2\,dx}\bigg)^{2^*_i}\\
&=\int_{\mathbb{R}^N}|\nabla V_0|^2\,dx -\sum_{i=1}^k \int_{\mathbb{R}^N}\int_{\mathbb{R}^N}\frac{V_0^{2^*_i}(x)V_0^{2^*_i}(y)}{|x-y|^{N-\alpha_i}}\,dxdy,\\
\end{split}
\]
that is, $V_0\in \mathcal{M}_{\mathbb{R}^N}$. This means that $V_0$ is a minimizer of $m(\mathbb{R}^N)$, the assertion follows.
\end{proof}

\bigskip

\section {The limit problem: uniqueness}

\bigskip

In this section, we prove the uniqueness of positive solution for the limit problem \eqref{eq:1.10} up to translations and dilations, the argument used is based on  the moving spheres method. Denote
$$
{w ^i}(x) = \int_{{\mathbb R^N}} {\frac{{u{{(y)}^{\frac{{N + \alpha_i }}
{{N - 2}}}}}}
{{|x - y{|^{N - {\alpha _i}}}}}dy}, \quad i= 1,\cdot\cdot\cdot, k,
$$
then \eqref{eq:1.10} can be written as
\begin{equation}\label{1.2}
 - \Delta u(x) = \sum\limits_{i = 1}^k {{w^i}(x)u{{(x)}^{\frac{{2 + {\alpha _i}}}
{{N - 2}}}}} \quad {\rm in }\quad \mathbb R^N.
\end{equation}
The moving spheres method can not be directly applied to $u$ since no decay law of $u$ at infinity is known.  So we turn to use the moving spheres method on the Kelvin transformation of $u$. The Kelvin transformations of $u$ and $w ^i$ are defined as
\begin{equation*}
v(x) = \frac{1}
{{|x{|^{N - 2}}}}u\big(\frac{x}
{{|x{|^2}}}\big),\ {z^i}(x) = \frac{1}
{{|x{|^{N - {\alpha _i}}}}}{w^i}\big(\frac{x}
{{|x{|^2}}}\big),\quad i= 1,\cdot\cdot\cdot, k.
\end{equation*}
We remark that the solution $u$ of \eqref{eq:1.10} belongs to $C^2(\mathbb R^N)$, this can be done by the same argument in \cite{YZZ}. Therefore,  $v(x)$ and $z^i(x)$ decay at infinity as $|x|^{2-N}$  and  $|x|^{\alpha _i-N}$ respectively. It can be verified that $v(x)$ and $z^i(x)$ satisfy
\begin{equation}\label{1.3}
\left\{ \begin{gathered}
 {z^i}(x) = \int_{{\mathbb R^N}} {\frac{{v{{(y)}^{\frac{{N + {\alpha _i}}}{{N - 2}}}}}}{{|x - y{|^{N - {\alpha _i}}}}}dy} \qquad \qquad\qquad  {\rm in} \ {\mathbb R^N}\backslash \{ 0\},  \\
  - \Delta v(x) = \sum\limits_{i = 1}^k {{z^i}(x)v{{(x)}^{\frac{{2 + {\alpha _i}}}{{N - 2}}}}}\qquad\quad\quad {\rm in}\ {\mathbb R^N}\backslash \{ 0\}.  \\
 \end{gathered} \right.
\end{equation}
We can also verify that the reflection functions
$$
 {v_\lambda }(x) = {(\frac{\lambda }{{|x|}})^{N - 2}}v(\frac{{{\lambda ^2}x}}{{|x{|^2}}}) = \frac{1}{{{\lambda ^{N - 2}}}}u(\frac{x}{{{\lambda ^2}}})
 $$
 and
 $$
 z_\lambda ^i(x) = {(\frac{\lambda }{{|x|}})^{N - {\alpha _i}}}z(\frac{{{\lambda ^2}x}}{{|x{|^2}}}) = \frac{1}{{{\lambda ^{N - {\alpha _i}}}}}{w^i}(\frac{x}{{{\lambda ^2}}}),
$$
of $v(x)$ and $z^i(x)$ with respect to $\partial B_\lambda(0)$  satisfy
\begin{equation}\label{11.4}
\left\{ \begin{gathered}
 z_\lambda ^i(x) = \int_{{\mathbb R^N}} {\frac{{{v_\lambda }{{(y)}^{\frac{{N + {\alpha _i}}}{{N - 2}}}}}}{{|x - y{|^{N - {\alpha _i}}}}}} dy\qquad\quad\quad\quad\qquad {\rm in}\ {\mathbb R^N}, \\
  - \Delta {v_\lambda }(x) = \sum\limits_{i = 1}^k {z_\lambda ^i(x)} {v_\lambda }{(x)^{\frac{{2 + {\alpha _i}}}{{N - 2}}}} \qquad\quad\quad {\rm in}\ {\mathbb R^N}.\\
 \end{gathered} \right.
\end{equation}

The spirit of the moving spheres method is to compare the value of functions with their reflections with respect to the sphere. For this purpose, we have the following result.
\begin{lemma}\label{Lemma 1.1}
Suppose that $z^i(x)$ and $z^i_\lambda(x)$ satisfy equation \eqref{1.3} and equation \eqref{11.4} respectively, then we have
\begin{equation}\label{11.5}
z_\lambda ^i(x) - {z^i}(x) = \int_{{B_\lambda }} {[\frac{1}{{|x - y{|^{N - {\alpha _i}}}}} - \frac{1}{{|x - \frac{{{\lambda ^2}y}}{{|y{|^2}}}{|^{N - {\alpha _i}}}}}{\big(\frac{\lambda }{|y|}\big)^{N - {\alpha _i}}}][{v_\lambda }{{(y)}^{\frac{{N + {\alpha _i}}}{{N - 2}}}} - v{{(y)}^{\frac{{N + {\alpha _i}}}{{N - 2}}}}]} dy.
\end{equation}
\end{lemma}
\begin{proof}
The result is a direct consequence of  \eqref{1.3} and  \eqref{11.4}. Indeed, we deduce from  \eqref{1.3} and  \eqref{11.4} that
\begin{equation*}
\begin{split}
  &z_\lambda ^i(x) - {z^i}(x)  \\
   &= \int_{{\mathbb{R}^N}} {\frac{{{v_\lambda }{{(y)}^{\frac{{N + {\alpha _i}}}
{{N - 2}}}}}}
{{|x - y{|^{N - {\alpha _i}}}}}} dy - \int_{{\mathbb{R}^N}} {\frac{{v{{(y)}^{\frac{{N + {\alpha _i}}}
{{N - 2}}}}}}
{{|x - y{|^{N - {\alpha _i}}}}}} dy  \\
   &= \int_{{B_\lambda }} {\frac{{{v_\lambda }{{(y)}^{\frac{{N + {\alpha _i}}}
{{N - 2}}}}}}
{{|x - y{|^{N - {\alpha _i}}}}}} dy + \int_{B_\lambda^C} {\frac{{{v_\lambda }{{(y)}^{\frac{{N + {\alpha _i}}}
{{N - 2}}}}}}
{{|x - y{|^{N - {\alpha _i}}}}}} dy  - \int_{{B_\lambda }} {\frac{{v{{(y)}^{\frac{{N + {\alpha _i}}}
{{N - 2}}}}}}
{{|x - y{|^{N - {\alpha _i}}}}}} dy - \int_{B_\lambda^C} {\frac{{v{{(y)}^{\frac{{N + {\alpha _i}}}
{{N - 2}}}}}}
{{|x - y{|^{N - {\alpha _i}}}}}} dy  \\
   &= \int_{{B_\lambda }} {\frac{{{v_\lambda }{{(y)}^{\frac{{N + {\alpha _i}}}
{{N - 2}}}}}}
{{|x - y{|^{N - {\alpha _i}}}}}} dy + \int_{{B_\lambda }} {\frac{{{v_\lambda }{{(\frac{{{\lambda ^2}y}}
{{|y{|^2}}})}^{\frac{{N + {\alpha _i}}}
{{N - 2}}}}}}
{{|x - \frac{{{\lambda ^2}y}}
{{|y{|^2}}}{|^{N - {\alpha _i}}}}}} \frac{{{\lambda ^{2N}}}}
{{|y{|^{2N}}}}dy  \\
   &- \int_{{B_\lambda }} {\frac{{v{{(y)}^{\frac{{N + {\alpha _i}}}
{{N - 2}}}}}}
{{|x - y{|^{N - {\alpha _i}}}}}} dy - \int_{{B_\lambda }} {\frac{{v{{(\frac{{{\lambda ^2}y}}
{{|y{|^2}}})}^{\frac{{N + {\alpha _i}}}
{{N - 2}}}}}}
{{|x - \frac{{{\lambda ^2}y}}
{{|y{|^2}}}{|^{N - {\alpha _i}}}}}\frac{{{\lambda ^{2N}}}}
{{|y{|^{2N}}}}} dy \hfill \\
   &= \int_{{B_\lambda }} {\frac{{{v_\lambda }{{(y)}^{\frac{{N + {\alpha _i}}}
{{N - 2}}}}}}
{{|x - y{|^{N - {\alpha _i}}}}}} dy + \int_{{B_\lambda }} {\frac{{v{{(y)}^{\frac{{N + {\alpha _i}}}
{{N - 2}}}}}}
{{|x - \frac{{{\lambda ^2}y}}
{{|y{|^2}}}{|^{N - {\alpha _i}}}}}} \frac{{{\lambda ^{N - {\alpha _i}}}}}
{{|y{|^{N - {\alpha _i}}}}}dy \\
   &- \int_{{B_\lambda }} {\frac{{v{{(y)}^{\frac{{N + {\alpha _i}}}
{{N - 2}}}}}}
{{|x - y{|^{N - {\alpha _i}}}}}} dy - \int_{{B_\lambda }} {\frac{{{v_\lambda }{{(y)}^{\frac{{N + {\alpha _i}}}
{{N - 2}}}}}}
{{|x - \frac{{{\lambda ^2}y}}
{{|y{|^2}}}{|^{N - {\alpha _i}}}}}{(\frac{\lambda }
{{|y|}})^{N - {\alpha _i}}}} dy  \\
  & = \int_{{B_\lambda }} {[\frac{1}
{{|x - y{|^{N - {\alpha _i}}}}}}  - \frac{1}
{{|x - \frac{{{\lambda ^2}y}}
{{|y{|^2}}}{|^{N - {\alpha _i}}}}}{(\frac{\lambda }
{{|y|}})^{N - {\alpha _i}}}][{v_\lambda }{(y)^{\frac{{N + {\alpha _i}}}
{{N - 2}}}} - {v_\lambda }{(y)^{\frac{{N + {\alpha _i}}}
{{N - 2}}}}]dy.
\end{split}
\end{equation*}
This completes the proof of Lemma \ref{Lemma 1.1}.
\end{proof}

\bigskip

To compare the values of $v(x)$ and $v_\lambda(x)$ in $B_\lambda(0)$, we need the following result.
\begin{lemma}\label{Lemma 1.2}
Suppose that $v(x)$ and $z^i(x)$ satisfy  \eqref{1.3} and \eqref{11.4} respectively, then we have the following inequality
\begin{equation}\label{11.9}
\begin{split}
  {\int_{{B_\lambda }} {|\nabla {{({v_\lambda }(x) - v(x))}^ + }|} ^2}dx &\le C[\sum\limits_{i = 1}^k {||} {v_\lambda }(x)||_{{L^{\frac{{2N}}
{{N - 2}}}}({B_\lambda })}^{\frac{{4 + 2{\alpha _i}}}
{{N - 2}}} + \sum\limits_{i = 1}^k {||} z_\lambda ^i|{|_{{L^{\frac{{2N}}
{{N - {\alpha _i}}}}}({B_\lambda })}} ||{v_\lambda }(x)||_{{L^{\frac{{2N}}
{{N - 2}}}}({B_\lambda })}^{\frac{{4 + {\alpha _i} - N}}
{{N - 2}}}]  \\
 &\times {\int_{{B_\lambda }} {|\nabla {{({v_\lambda }(x) - v(x))}^ + }|} ^2}dx.
\end{split}
\end{equation}
\end{lemma}

\begin{proof}To deal with the possible singularity of $v$ at zero, we first define a cut-off function
\begin{equation*}
{\eta _\varepsilon }(x) = {\eta _\varepsilon }(|x|) = \left\{ \begin{gathered}
  0 \qquad \qquad {\rm for}\  x \in {B_\varepsilon }, \hfill \\
  1\qquad \qquad {\rm for}\ x \in B_{2\varepsilon }^c \hfill \\
\end{gathered}  \right.
\end{equation*}
with $|\nabla \eta | \le \frac{2}
{\varepsilon }$ in ${B_{2\varepsilon }}\backslash {B_\varepsilon }$.
Let
\begin{equation*}
  \varphi _\varepsilon (x) = \eta _\varepsilon ^2{({v_\lambda }(x) - v(x))^+,  \quad \psi  _\varepsilon }(x) = \eta _\varepsilon ({v_\lambda }(x) - v(x))^ +,
\end{equation*}
then we have
$$
  |\nabla {\psi _\varepsilon }(x)|^2 = \nabla {({v_\lambda }(x) - v(x)) }\nabla {\varphi _\varepsilon }{\text{ + [}}{({v_\lambda }(x) - v(x))^ + }{]^2}|\nabla {\eta _\varepsilon }{|^2}.
  $$
 We estimate
  \begin{equation}\label{11.6}
\begin{split}  &\int_{{B_\lambda }\setminus {B_{2\varepsilon }}} {|\nabla {{({v_\lambda }(x) - v(x))}^ + }} {|^2}dx  \\
   &\le \int_{{B_\lambda }} {|\nabla {\psi _\varepsilon }(x){|^2}} dx \hfill \\
   &= \int_{{B_\lambda }} {\nabla {{({v_\lambda }(x) - v(x))} }} \nabla {\varphi _\varepsilon }(x)dx + \int_{{B_\lambda }} {{[{({v_\lambda }(x) - v(x))}^ +]^2 }} |\nabla {\eta _\varepsilon }{|^2}dx \\
  & =  - \int_{{B_\lambda }} {\Delta ({v_\lambda }(x) - v(x))} {\varphi _\varepsilon }(x)dx + {I_\varepsilon } \\
   &\le \int_{{B_\lambda }} {\sum\limits_{i = 1}^k {[z_\lambda ^i(x){v_\lambda }{{(x)}^{\frac{{2 + {\alpha _i}}}
{{N - 2}}}} - {z^i}(x)v{{(x)}^{\frac{{2 + {\alpha _i}}}
{{N - 2}}}}]} {{({v_\lambda }(x) - v(x))}^ + }} dx + {I_\varepsilon },
\end{split}
\end{equation}
where $I_\varepsilon=\int_{{B_\lambda }} {{[{({v_\lambda }(x) - v(x))}^ +]^2 }} |\nabla {\eta _\varepsilon }{|^2}dx.$
We claim that $I_\varepsilon\to 0$ as $\varepsilon \to 0$. In fact,
  \begin{equation*}
\begin{split}  {I_\varepsilon } &\le {\Big(\int_{{B_{2\varepsilon }}} {[{{({v_\lambda }(y) - v(x))}^ + }} ]^{\frac{{2N}}
{{N - 2}}}}dx{\Big)^{\frac{{N - 2}}
{N}}}  {\Big(\int_{{B_{2\varepsilon }}} {|\nabla \eta {|^N}dx} \Big)^{\frac{2}
{N}}}\\
&\leq \Big(\int_{B_{\frac{2\varepsilon}{\lambda^2}}}u(x)^{\frac{2N}{N-2}}\,dx\Big)
^{\frac{{N - 2}}
{N}}  {\Big(\int_{{B_{2\varepsilon }}} {|\nabla \eta {|^N}dx} \Big)^{\frac{2}
{N}}}\to 0\\
\end{split}
\end{equation*}
as $\varepsilon  \to 0$.

Next, we estimate $z_\lambda ^i(x){v_\lambda }{(x)^{\frac{{2 + {\alpha _i}}}
{{N - 2}}}} - {z^i}(x)v{(x)^{\frac{{2 + {\alpha _i}}}
{{N - 2}}}} $
in $B_\lambda$. If $z_\lambda ^i(x) \le {z^i}(x)$, apparently we have
\begin{equation}\label{11.7}
z_\lambda ^i(x){v_\lambda }{(x)^{\frac{{2 + {\alpha _i}}}
{{N - 2}}}} - {z^i}(x)v{(x)^{\frac{{2 + {\alpha _i}}}
{{N - 2}}}} \le z_\lambda ^i(x)[{v_\lambda }{(x)^{\frac{{2 + {\alpha _i}}}
{{N - 2}}}} - v{(x)^{\frac{{2 + {\alpha _i}}}
{{N - 2}}}}].
\end{equation}
While if $z_\lambda ^i(x) > {z^i}(x)$, there holds
\begin{equation}\label{11.8}
\begin{split}
  &z_\lambda ^i(x){v_\lambda }{(x)^{\frac{{2 + {\alpha _i}}}
{{N - 2}}}} - {z^i}(x)v{(x)^{\frac{{2 + {\alpha _i}}}
{{N - 2}}}}  \\
&=[z_\lambda ^i(x) - {z^i}(x)]{v_\lambda }{(x)^{\frac{{2 + {\alpha _i}}}
{{N - 2}}}} + z ^i(x)[{v_\lambda }{(x)^{\frac{{2 + {\alpha _i}}}
{{N - 2}}}} - v{(x)^{\frac{{2 + {\alpha _i}}}
{{N - 2}}}}]\\
&\leq [z_\lambda ^i(x) - {z^i}(x)]^+{v_\lambda }{(x)^{\frac{{2 + {\alpha _i}}}
{{N - 2}}}} + z ^i(x)[{v_\lambda }{(x)^{\frac{{2 + {\alpha _i}}}
{{N - 2}}}} - v{(x)^{\frac{{2 + {\alpha _i}}}
{{N - 2}}}}]^+\\
   &\le [z_\lambda ^i(x) - {z^i}(x)]^+{v_\lambda }{(x)^{\frac{{2 + {\alpha _i}}}
{{N - 2}}}} + z_\lambda ^i(x)[{v_\lambda }{(x)^{\frac{{2 + {\alpha _i}}}
{{N - 2}}}} - v{(x)^{\frac{{2 + {\alpha _i}}}
{{N - 2}}}}]^+ .
\end{split}
\end{equation}

Substituting  \eqref{11.7} and  \eqref{11.8} into  \eqref{11.6},  we get
\begin{equation*}
\begin{gathered}
 {\int_{{B_\lambda }\backslash {B_{2\varepsilon }}} {|\nabla {{({v_\lambda }(x) - v(x))}^ + }|} ^2}dx \hfill \\
   \le \sum\limits_{i = 1}^k {\int_{{B_\lambda }} {{{[z_\lambda ^i(x) - {z^i}(x)]}^ + }{v_\lambda }{{(x)}^{\frac{{2 + {\alpha _i}}}
{{N - 2}}}}{{({v_\lambda }(x) - v(x))}^ + }} dx}  \hfill \\
   + \sum\limits_{i = 1}^k {\int_{{B_\lambda }} {z_\lambda ^i(x){v_\lambda }{{(x)}^{\frac{{4 + {\alpha _i} - N}}
{{N - 2}}}}{{[{{({v_\lambda }(x) - v(x))}^ + }]}^2}} dx}  + {I_\varepsilon } \hfill \\
   \le \sum\limits_{i = 1}^k {\int_{{B_\lambda }}\int_{{B_\lambda }} {\frac{{{v_\lambda }{{(y)}^{\frac{{2 + {\alpha _i}}}
{{N - 2}}}}}}
{{|x - y{|^{N - {\alpha _i}}}}}{{[{v_\lambda }(y) - v(y)]}^ + }  {v_\lambda }{{(x)}^{\frac{{2 + {\alpha _i}}}
{{N - 2}}}}{{[{v_\lambda }(x) - v(x)]}^ + }} dx} dy \hfill \\
   + \sum\limits_{i = 1}^k {\int_{{B_\lambda }} {z_\lambda ^i(x){v_\lambda }{{(x)}^{\frac{{4 + {\alpha _i} - N}}
{{N - 2}}}}{{[{{({v_\lambda }(x) - v(x))}^ + }]}^2}dx} }  + {I_\varepsilon } \hfill \\
   \le \sum\limits_{i = 1}^k {||{v_\lambda }{{(y)}^{\frac{{2 + {\alpha _i}}}
{{N - 2}}}}{{({v_\lambda }(y) - v(y))}^ + }||} _{{L^{\frac{{2N}}
{{N + {\alpha _i}}}}}({B_\lambda })}^2 \hfill \\
   + {\sum\limits_{i = 1}^k {||z_\lambda ^i(x)||} _{{L^{\frac{{2N}}
{{N - \alpha_i}}}}({B_\lambda })}}  ||{v_\lambda }(x)||_{{L^{\frac{{2N}}
{{N - 2}}}}({B_\lambda })}^{\frac{{4 + {\alpha _i} - N}}
{{N - 2}}}||{({v_\lambda }(y) - v(y))^ + }||_{{L^{\frac{{2N}}
{{N - 2}}}}({B_\lambda })}^2 + {I_\varepsilon } \hfill \\
   \le \sum\limits_{i = 1}^k {||{v_\lambda }(x)||_{{L^{\frac{{2N}}
{{N - 2}}}}({B_\lambda })}^{\frac{{4 + 2{\alpha _i} }}
{{N - 2}}}||{{({v_\lambda }(x) - v(x))}^ + }||_{{L^{\frac{{2N}}
{{N - 2}}}}({B_\lambda })}^2}  \hfill \\
   + {\sum\limits_{i = 1}^k {||z_\lambda ^i(x)||} _{{L^{\frac{{2N}}
{{N - \alpha_i}}}}({B_\lambda })}}  ||{v_\lambda }(x)||_{{L^{\frac{{2N}}
{{N - 2}}}}({B_\lambda })}^{\frac{{4 + {\alpha _i} - N}}
{{N - 2}}}||{({v_\lambda }(x) - v(x))^ + }||_{{L^{\frac{{2N}}
{{N - 2}}}}({B_\lambda })}^2 + {I_\varepsilon }. \hfill \\
\end{gathered}
\end{equation*}
Using Sobolev inequality and letting $\varepsilon  \to 0$, we obtain
\begin{equation*}
\begin{split}
  {\int_{{B_\lambda }} {|\nabla {{({v_\lambda }(x) - v(x))}^ + }|} ^2}dx &\le C[\sum\limits_{i = 1}^k {||} {v_\lambda }(x)||_{{L^{\frac{{2N}}
{{N - 2}}}}({B_\lambda })}^{\frac{{4 + 2{\alpha _i}}}
{{N - 2}}} + \sum\limits_{i = 1}^k {||} z_\lambda ^i|{|_{{L^{\frac{{2N}}
{{N - {\alpha _i}}}}}({B_\lambda })}} ||{v_\lambda }(x)||_{{L^{\frac{{2N}}
{{N - 2}}}}({B_\lambda })}^{\frac{{4 + {\alpha _i} - N}}
{{N - 2}}}]  \\
 &\times {\int_{{B_\lambda }} {|\nabla {{({v_\lambda }(x) - v(x))}^ + }|} ^2}dx,
\end{split}
\end{equation*}
this completes the proof of Lemma \ref{Lemma 1.2}.
\end{proof}

\bigskip

Lemmas \ref{Lemma 1.1} and \ref{Lemma 1.2} are the maximum principles in an integral form, they can be used to replace the usual maximum principles in a differential form. In fact, if
$$
C[\sum\limits_{i = 1}^k {||} {v_\lambda }(x)||_{{L^{\frac{{2N}}
{{N - 2}}}}({B_\lambda })}^{\frac{{4 + 2{\alpha _i}}}
{{N - 2}}} + \sum\limits_{i = 1}^k {||} z_\lambda ^i|{|_{{L^{\frac{{2N}}
{{N - {\alpha _i}}}}}({B_\lambda })}} ||{v_\lambda }(x)||_{{L^{\frac{{2N}}
{{N - 2}}}}({B_\lambda })}^{\frac{{4 + {\alpha _i} - N}}
{{N - 2}}}]<1,
$$
we may infer from equation \eqref{11.9} that $v_\lambda(x)\leq v(x)$ in $B_\lambda\setminus \{0\}$.  By Lemma \ref{Lemma 1.1}, we have $z^i_\lambda(x)\leq z^i(x)$ in $B_\lambda\setminus \{0\}$. The conclusion is the same as that of the maximum principle.

Now, we  show that the spheres can be moved from infinity.
\begin{lemma}\label{Lemma 1.3}
There exists $\lambda>0$ large enough, such that
$$
  {v_\lambda }(x) \le v(x)\qquad {\rm in}\ {B_\lambda\setminus\{0\} }
  $$
  and
  $$
  z_\lambda ^i(x) \le {z^i}(x)\qquad {\rm in}\ {B_\lambda\setminus\{0\} } \ ,i=1,2,...,k.
$$
\end{lemma}
\begin{proof}
By the definition of $v_\lambda$, we have
$$
  \int_{{B_\lambda }} {{v_\lambda }{{(x)}^{\frac{{2N}}
{{N - 2}}}}} dx = \int_{{B_{\frac{1}
{\lambda }}}} {u{{(y)}^{\frac{{2N}}
{{N - 2}}}}} dy \to 0\ \rm{as}\ \lambda  \to  + \infty
$$
and
$$
  \int_{{B_\lambda }} {z_\lambda ^i{{(x)}^{\frac{{2N}}
{{N - {\alpha _i}}}}}} dx = \int_{{B_{\frac{1}
{\lambda }}}} {{w^i}{{(y)}^{\frac{{2N}}
{{N - {\alpha _i}}}}}} dy \to 0\ \rm{as}\ \lambda  \to  + \infty
$$ for $i=1,2,...,k$.

Hence, we can choose $\lambda$ large enough, such that
\begin{equation*}
\begin{gathered}
 C[\sum\limits_{i = 1}^k {||{v_\lambda }(x)||_{{L^{\frac{{2N}}
{{N - 2}}}}({B_\lambda })}^{\frac{{4 + 2{\alpha _i}}}
{{N - 2}}} + } \sum\limits_{i = 1}^k {||z_\lambda ^i(x)|{|_{{L^{\frac{{2N}}
{{N - 2}}}}({B_\lambda })}}} ||{v_\lambda }(x)||_{{L^{\frac{{2N}}
{{N - 2}}}}({B_\lambda })}^{\frac{{4 + {\alpha _i} - N}}
{{N - 2}}}] < \frac{1}{2},\hfill \\
\end{gathered}
\end{equation*}
and Lemma \ref{Lemma 1.2} implies that $v_\lambda(x) \le v(x)$ in ${B_\lambda }\backslash \{ 0\} $. Moreover, we infer from Lemma \ref{Lemma 1.1} that $z_\lambda ^i(x) \le {z^i}(x)$ in ${B_\lambda }\backslash \{ 0\} $ for $i=1,2,...,k$.
\end{proof}
\bigskip
We note that the above process also works for
$${u^b}(x) = u(x + b)$$
with $b \in {\mathbb{R}^N}$. In particular, for any $b
\in {\mathbb{R}^N}$ and $\lambda$ large enough, we have
$$
  {v_\lambda^b }(x) \le v^b(x)\quad\qquad {\rm in}\ {B_\lambda }\backslash \{ 0\}
$$
and
$$
  z_\lambda ^{b,i}(x) \le {z^{b,i}}(x)\qquad \ {\rm in}\ {B_\lambda }\backslash \{ 0\}\  {\rm for} \ i=1,2,...,k,
$$
where $v^b(x)$ is the Kelvin transformation of $u^b(x)$, and $z_\lambda ^{b,i}(x)$ is defined in a similar way.

Let
\[{ \lambda_b} = \inf \{ \lambda |v_\mu ^b(x) \le {v^b}(x),z_\mu ^{b,i}(x) \le {z^{b,i}}(x)(i=1,2,...,k)\ {\rm in}\ {B_\mu }\setminus \{ 0\} \ {\rm with}\ \lambda  < \mu  <  + \infty \}. \]
\begin{lemma}\label{Lemma 1.4}
There exists $\bar b \in \mathbb{R}^N$, such that  $ \lambda_{\bar b}>0 $.
\end{lemma}
Before proving Lemma \ref{Lemma 1.4}, we establish a technical lemma.
\begin{lemma}\label{t 3.5}
Suppose $u\in C^1(\mathbb{R}^N)$, if for all $b\in \mathbb{R}^N$ and $\lambda>0$, the following inequality holds
 \begin{equation}\label{3.31}
\frac{1}{|x|^{N-2}}u_b\big(\frac{x}{|x|^2}\big)-\frac{1}{\lambda^{N-2}}u_b\big(\frac{x}{\lambda^2}\big)\geq0,\ \forall x\in B_\lambda\setminus\{0\},
\end{equation}
then we have $u(x)\equiv C$, where $u_b(x)=u(x+b)$.

\end{lemma}
\begin{proof}
Let $g_{b,\lambda}(x)=\frac{1}{|x|^{N-2}}u_b(\frac{x}{|x|^2})-\frac{1}{\lambda^{N-2}}u_b(\frac{x}{\lambda^2})$. Then,
\begin{equation}
\begin{split}
g_{b,|x|}(x)=\frac{1}{|x|^{N-2}}u_b(\frac{x}{|x|^2})-\frac{1}{|x|^{N-2}}u_b(\frac{x}{|x|^2})=0
\end{split}
\end{equation}
and for any $0<r<1$, we infer from equation \eqref{3.31} that
\begin{equation}
\begin{split}
g_{b,|x|}(rx)
=\frac{1}{|rx|^{N-2}}u_b(\frac{rx}{|rx|^2})-\frac{1}{|x|^{N-2}}u_b(\frac{rx}{|x|^2})\geq0.
\end{split}
\end{equation}
This implies
\begin{equation}
\begin{split}
\frac{dg_{b,|x|}(rx)}{dr}|_{r=1}\leq0,
\end{split}
\end{equation}
that is,
\begin{equation}
\begin{split}
(N-2)\frac{1}{|x|^{N-2}}u_b(\frac{x}{|x|^2})+\frac{2}{|x|^N}\nabla u_b(\frac{x}{|x|^2}) x\geq0.
\end{split}
\end{equation}
In other word,
\begin{equation}
\begin{split}
(N-2)u(y+b)+2y\cdot \nabla u(y+b)\geq0,
\end{split}
\end{equation}
namely,
\begin{equation}\label{3.5a}
\begin{split}
(N-2)u(x)+2(x-b)\nabla u(x)\geq0.
\end{split}
\end{equation}
Dividing  both sides of \eqref{3.5a} by $|b|$ and letting $|b|\to\infty$, we find $\nabla u(x)\frac{b}{|b|}\leq0$. Since $b$ is arbitrary, we have $\nabla u(x)=0$ or $u\equiv C$.
The assertion follows.
\end{proof}

\bigskip

\begin{proof}[Proof of Lemma \ref{Lemma 1.4}]
Suppose on the contrary that ${\lambda _{b}} \equiv 0$ for all $b \in \mathbb{R}^N$, then we infer from Lemma \ref{t 3.5} that $v(x) \equiv 0$, which yields
 $u \equiv 0$. This is a contradiction because $u$ is a positive solution.
\end{proof}

\bigskip

\begin{lemma}\label{Lemma 1.5}
If ${\lambda _{ b}} > 0$, then $v_{{{ \lambda  }_{ b }}}^b(x) \equiv {v^b}(x),z_{{{  \lambda  }_{ b }}}^{ b ,i}(x) \equiv {z^{b,i}}(x)(i=1,2,...,k)$ in ${B_{{{ \lambda  }_{ b }}}}\setminus \{ 0\}$.
\end{lemma}
\begin{proof}
Without loss of generality we assume $b=0$  and  denote
 $$v_{\lambda_0}(x)=v_{{{\lambda  }_0}}^0(x),{z^i_{\lambda_0}}(x) = z_{{{\lambda  }_{ 0 }}}^{ 0 ,i}(x),\,i=1,2,...,k.$$

Assume on the contrary, we claim that there exists a constant $\gamma>0$ such that $v(x)-v_{\lambda_0}(x)\geq \gamma$ in $B_{\frac{\lambda_0}2}\setminus \{0\}$. Indeed, by the strong maximum principle, we have  ${v_{\lambda_0  }}(x) < v(x)$ in ${B_{\lambda_0}}(0)\setminus \{ 0\}$. In particular, $\gamma  = \mathop {\inf }\limits_{\partial {B_{\frac{{\lambda_0  }}
{2}}}} [v(x) - {v_{\lambda_0  }}(x)] > 0$.

Set $\varphi (x) = \gamma (1 - \frac{{{r^{N - 2}}}}{{|x{|^{N - 2}}}})$ with $0<r<\frac{\lambda_0}{2}$ and $\psi(x)=[v(x)-{v_{\lambda_0  }}(x)]-\varphi (x) $, then we have
$$
\left\{
  \begin{array}{ll}
  \displaystyle
  \psi(x) > 0 &{\rm on}\  \partial B_{\frac{\lambda_0}2},  \\
  \displaystyle
 \\ \psi(x){|_{\partial {B_r}}} \ge 0  &{\rm on}\  \partial B_{r},\\
    \displaystyle
 \\  - \Delta \psi(x) \ge 0 &{\rm in}\  \partial B_{\frac{\lambda_0}{2}}\setminus B_{r}.
 \end{array}
\right.$$
We infer from the maximum principle that
$$
v(x)-{v_{\lambda_0  }}(x)\ge  \gamma (1 - \frac{{{r^{N - 2}}}}{{|x{|^{N - 2}}}}).
$$
Let $r\to 0$, then we get $v(x)-{v_{\lambda_0  }}(x)\ge  \gamma$ in ${B_{\frac{{\lambda_0  }}
{2}}}\setminus \{ 0\}$. This proves the claim.

If the conclusion of this lemma does not hold, then there exists a sequence ${\lambda_n}$ with $\lambda_n<\lambda_0,\lambda_n \to \lambda_0$ and $\mathop {\inf }\limits_{{B_{{\lambda _n}}}\setminus \{ 0\} } [v(x) - {v_{{\lambda _n}}}(x)] < 0$. We deduce from the above claim that
 $v(x) - {v_{{\lambda _n}}}(x) \ge \frac{\gamma }
{2} > 0$
for $n$ large enough. Hence,
$\mathop {\inf }\limits_{{B_{{\lambda _n}}}\setminus \{ 0\} } [v(x) - {v_{{\lambda _n}}}(x)]$ is attained at some $x_n$ with $\frac{\lambda_0}{2}<|x_n|<\lambda_n$. In particular, we have $v({x_n}) - {v_{{\lambda _n}}}({x_n}) < 0$ and $\nabla (v({x_n}) - {v_{{\lambda _n}}}({x_n})) = 0.$

We can assume that, up to a subsequence, $x_n\to x_0$, then we conclude that $x_0\in \partial {B_{\lambda_0  }}
$ and $\nabla (v({x_0}) - {v_{\lambda_0}}({x_0})) = 0$. However, this contradicts to the Hopf Lemma.
\end{proof}

\bigskip

\begin{lemma}\label{Lemma 1.6}
For all $b\in \mathbb R^N$, we have $\lambda_b>0$.
\end{lemma}
\begin{proof}
Since we have proved $\lambda_{\bar b}>0$ for some $\bar b\in \mathbb R^N$, it follows from Lemma \ref{Lemma 1.5} that
\[{v_{\bar b }}(x) = \frac{{{{\lambda_{\bar b } }^{N - 2}}}}
{{|x{|^{N - 2}}}}{v_{\bar b }}\Big(\frac{{{{\lambda_{\bar b }^2 }}x}}
{{|x{|^2}}}\Big)\] for $x\neq 0$.
That is
$$
\frac{\lambda_{\bar b}^{N-2}}{|x|^{N-2}}u_{\bar b}\Big(\frac x{|x|^2}\Big)=u_{\bar b}\Big(\frac x{\lambda_{\bar b}^2}\Big).
$$
for $x\neq 0$.
Let $y=\frac x{|x|^2}$ and $|y|\to \infty$,  we get
\begin{equation}\label{11.11}
\lim_{|y|\to \infty}|y|^{N-2} u_{\bar b}(y)=\frac{1}{\lambda_{\bar b}^{N-2}}u_{\bar b}(0).
\end{equation}

On the other hand, if the conclusion of the lemma does not hold, then there exists $b\in \mathbb R^N$, such that
$${v_b}(x) \ge \frac{{{\lambda ^{N - 2}}}}
{{|x{|^{N - 2}}}}{v_b}\Big(\frac{{{\lambda ^2}x}}
{{|x{|^2}}}\Big)$$ for all $\lambda >0$ and $x\in B_\lambda\setminus \{ 0\}
$. That is
$$
\frac{1}{|x|^{N-2}}u_b(\frac x{|x|^2})\geq \frac {1}{\lambda^{N-2}}u_b\Big(\frac x{\lambda^2}\Big)
$$for all $\lambda >0$ and $x\in B_\lambda\setminus \{ 0\}
$. Letting $x\to 0$ and using \eqref{11.11}, we obtain
$$
\frac{\lambda^{N-2}}{\lambda_{\bar b}^{N-2}}u_{\bar b}(0)\geq u_b(0)
$$
for all $\lambda>0$, which is impossible for $\lambda$ small.

\end{proof}

\begin{lemma}\label{Lemma 1.7}
The positive solution $u$ of problem \eqref{eq:1.10} must have the following form
$$
u(x) = C{\Big(\frac{\varepsilon }
{{{\varepsilon ^2} + |x - {x_0}{|^2}}}\Big)^{\frac{{N - 2}}
{2}}}
$$
for $\varepsilon>0$.
\end{lemma}
\begin{proof}
 By Lemmas \ref{Lemma 1.5} and  \ref{Lemma 1.6}, for all $b \in {\mathbb{R}^N}$, there holds
\[u(x) = \frac{{{{\lambda  }_b}^{N - 2}}}
{{|x - b{|^{N - 2}}}}u\Big(\frac{{{\lambda  _b^2}(x - b)}}
{{|x - b{|^2}}} + b\Big),\]
which implies
\[\mathop {\lim }\limits_{|x| \to \infty } |x{|^{N-2}}u(x) = \lambda  _b^{N-2}u(b) = \lambda  _0^{N-2}u(0) = B > 0.\]

If $B=1$, we have
\begin{equation}\label{11.12}
u(x) = \frac{\lambda  _0^{N - 2}}
{{|x|^{N - 2}}}[u(0) + \nabla u(0)\frac{{\lambda _0^2x}}
{{|x{|^2}}} + o(\frac{1}
{{|x|}})]
\end{equation}and
\begin{equation}\label{11.13}
u(x) = \frac{{\lambda  _b^{N - 2}}}
{{|x - b{|^{N - 2}}}}[u(b) + \nabla u(b)\frac{{\lambda _b^2(x - b)}}
{{|x - b{|^2}}} + o(\frac{1}
{{|x - b|}})]
\end{equation}
 as $x \to  + \infty$.
Replacing
 $$
 \frac 1{|x-b|^{N-2}}=\frac 1{|x|^{N-2}}+(N-2)\frac {x\cdot b}{|x|^N}+O(\frac 1{|x|^N})
 $$
 and
 $$
 \frac 1{|x-b|^{2}}=\frac 1{|x|^2}+2\frac {x\cdot b}{|x|^4}+O(\frac 1{|x|^4}),
 $$
into \eqref{11.13} we find
\begin{equation}\label{11.14}
u(x)=\lambda_b^{N-2}[\frac{u(b)}{|x|^{N-2}}+(N-2)\frac{x\cdot b u(b)}{|x|^N}+\frac{\nabla u(b)\cdot x\lambda_b^2}{|x|^N}+O(\frac{1}{|x|^N})].
\end{equation}
Comparing the coefficients of $\frac{1}{|x|^{N-2}}$ and $\frac{x}{|x|^N}$ in  \eqref{11.12} and  \eqref{11.14}, we obtain
$$
\lambda_0^{N-2}u(0)=\lambda_b^{N-2}u(b)
$$
and
\begin{equation}\label{11.15}
\nabla u(0)\lambda_0^N=\nabla u(b)\lambda_b^N+(N-2)bu(b)\lambda_b^{N-2}.
\end{equation}
It follows from $B=1$ that $\lambda_b=u(b)^{-\frac 1{N-2}}$ and $\lambda_0=u(0)^{-\frac 1{N-2}}$, hence \eqref{11.15} can be written as
\[{u^{ - \frac{N}
{{N - 2}}}}(b)\frac{{\partial u}}
{{\partial {x_i}}}(b) = {u^{ - \frac{N}
{{N - 2}}}}(0)\frac{{\partial u}}
{{\partial {x_i}}}(0) - (N - 2)b_i.\]
Therefore, we have
\[\begin{gathered}
 {u^{ - \frac{2}
{{N - 2}}}}(x) = |x - {x_0}{|^2} + d \hfill \\
\end{gathered} \]
or
\[\begin{gathered}
u(x) = {(\frac{1}
{{d + |x - {x_0}{|^2}}})^{\frac{{N - 2}}
{2}}}. \hfill \\
\end{gathered} \]

If $B\not=1$, the same argument shows that
$$
u(x) = C  {(\frac{\varepsilon }
{{\varepsilon^2  + |x - {x_0}{|^2}}})^{\frac{{N - 2}}
{2}}}
$$
for $\varepsilon>0$.

\end{proof}

\bigskip

\bigskip

 \section{A positive solution}

\bigskip

In this section, we show that there is a positive solution of \eqref{eq:1.1}  at the mountain pass value
\begin{equation}\label{eq:3.1}
c_{\lambda,\Omega}= \inf_{\gamma\in\Gamma}\max_{t\in[0,1]}I_{\lambda,\Omega}(\gamma(t)),
\end{equation}
where
\[
\Gamma:=\{\gamma\in C([0,1], H^1_0(\Omega)): \gamma(0) =0, I_{\lambda,\Omega}(\gamma(1)) \leq 0, \gamma(1)\not = 0\}.
\]
Let
\begin{equation} \label{eq:3.2}
\mathcal{N}_{\lambda,\Omega}=\{u\in H^1_0(\Omega)\setminus \{0\}:\langle I'_{\lambda,\Omega}(u), u\rangle = 0\}.
\end{equation}
We may verify that $\mathcal{N}_{\lambda,\Omega}$ is a manifold, and for each $u\in H^1_0(\Omega)$, there is a unique $t_u>0$ such that  $t_u u\in\mathcal{N}_{\lambda,\Omega}$.
Define
\begin{equation} \label{eq:3.3}
m_{\lambda,\Omega}=\inf\{I_{\lambda,\Omega}(u): u\in \mathcal{N}_{\lambda,\Omega}\}
\end{equation}
and
\begin{equation} \label{eq:3.4}
c^s_{\lambda,\Omega}=\inf_{u\in H^1_0(\Omega)\setminus \{0\}}\sup_{t>0}I_{\lambda,\Omega}(tu).
\end{equation}
Arguing as the proof of Theorem 4.2 in \cite{W}, we see that
\begin{equation} \label{eq:3.5}
m_{\lambda,\Omega} = c_{\lambda,\Omega}=c^s_{\lambda,\Omega}.
\end{equation}

It was proved in \cite{GY} the following Br\'{e}zis-Lieb type lemma.
\begin{lemma}\label{lem:3.1} Let $N\geq 3$ and $0<\alpha<N$. If $\{u_n\}$ is a bounded sequence in $L^{\frac{2N}{N-2}}(\mathbb{R}^N)$ and
$u_n\to u$ almost everywhere in $\mathbb{R}^N$ as $n\to\infty$. Then we have
\[
\begin{split}
&\lim_{n\to\infty}\int_{\mathbb{R}^N}(|x|^{-(N-\alpha)}*|u_n|^{2^*_\alpha})|u_n|^{2^*_\alpha}\,dx-\int_{\mathbb{R}^N}(|x|^{-(N-\alpha)}*|u_n-u|^{2^*_\alpha})|u_n-u|^{2^*_\alpha}\,dx\\ &=\int_{\mathbb{R}^N}(|x|^{-(N-\alpha)}*|u|^{2^*_\alpha})|u|^{2^*_\alpha}\,dx.\\
\end{split}
\]
\end{lemma}
\bigskip

We derive from Lemma \ref{lem:3.1} the following result.
\begin{lemma}\label{lem:3.2}
The functional $I_{\lambda,\Omega}$ satisfies $(PS)_c$ condition if $c<m(\mathbb{R}^N).$
\end{lemma}

\begin{proof} Let $\{u_n\}\subset H^1_0(\Omega)$ be a  $(PS)_c$ sequence of $I_{\lambda,\Omega}$ with $c<m(\mathbb{R}^N)$, that is,
\[
I_{\lambda,\Omega}(u_n)\to c,\quad I'_{\lambda,\Omega}(u_n)\to 0
\]
as $n\to\infty$. This yields
\[
c+o(1)= (\frac 12-\frac 1p)\lambda\int_\Omega u_n^p\,dx +\sum_{i=1}^k(\frac 12 -\frac 1{22^*_i})\int_\Omega\int_\Omega\frac{u_n^{2_i^*}(x)u_n^{2_i^*}(y)}{|x-y|^{N-\alpha_i}}\,dxdy
\]
implying that there is a positive constant $C$ such that for $i=1,2,\cdot\cdot\cdot,k$,
\[
\int_\Omega u_n^p\,dx\leq C, \quad \int_\Omega\int_\Omega\frac{u_n^{2_i^*}(x)u_n^{2_i^*}(y)}{|x-y|^{N-\alpha_i}}\,dxdy\leq C,
\]
and then $\{u_n\}$ is bounded in $H^1_0(\Omega)$. So we may assume that
\begin{equation} \label{eq:3.6}
u_n\rightharpoonup u_0\quad{\rm in}\quad H^1_0(\Omega);\quad u_n\to u_0\quad a.e. \quad{\rm in}\quad \Omega;\quad u_n\to u_0\quad{\rm in}\quad L^p(\Omega).
\end{equation}
The weak convergence implies that $I'_{\lambda,\Omega}(u_0) =0$, namely, $u_0$ is a weak solution of \eqref{eq:1.1}. Moreover, $I_{\lambda,\Omega}(u_0)\geq 0$. Let $v_n=u_n-u_0$. By Lemma \ref{eq:3.1} and \eqref{eq:3.6}, we deduce that
\begin{equation} \label{eq:3.7}
o(1) = \langle I'_{\lambda,\Omega}(v_n),v_n \rangle +\langle I'_{\lambda,\Omega}(u_0), u_0\rangle = \langle I'_{\lambda,\Omega}(v_n),v_n\rangle
\end{equation}
and
\begin{equation} \label{eq:3.8}
c+o(1) = I_{\lambda,\Omega}(v_n) + I_{\lambda,\Omega}(u_0)\geq I_{\lambda,\Omega}(v_n).
\end{equation}

Suppose now on the contrary that $u_n$ does not converge strongly to $u_0$ in $H^1_0(\Omega)$, then $v_n=u_n-u_0\not\to 0$ in $H^1_0(\Omega)$ as $n\to\infty$. Extending $v_n$ to $\mathbb{R}^N$ by setting $v_n=0$ outside $\Omega$. Then, there is $t_n>0, \, t_n\to 1$ as $n\to\infty$  such that $t_nv_n\in \mathcal{M}_{\mathbb{R}^N}$. We derive from \eqref{eq:3.7} and \eqref{eq:3.8} that $c\geq m(\mathbb{R}^N)$ a contradiction. The assertion follows.
\end{proof}

\bigskip

Now, we verify the condition in Lemma \ref{lem:3.2}. Without loss of generality, we assume $B_{2r}(0)\subset \Omega$ for $r>0$ small.  Let $\varphi\in C^1_0(B_{2r}(0))$ be
a cut-off function satisfying $\varphi\equiv 1$ on $B_r(0)$. Define
\begin{equation} \label{eq:3.9}
U_\varepsilon(x) =\varepsilon^{-\frac{N-2}2}U\big(\frac x\varepsilon\big),\quad u_\varepsilon(x) =\varphi(x)U_\varepsilon(x),
\end{equation}
where  $U(x) = \frac C{(1+|x|^2)^{\frac{N-2}2}}$ is the minimizer of $m(\mathbb{R}^N)$.

\begin{lemma}\label{lem:3.3}
There holds $c_{\lambda,\Omega}<m(\mathbb{R}^N).$
\end{lemma}

\begin{proof} We may verify as \cite{GY} that
\begin{equation} \label{eq:3.10}
\int_\Omega|\nabla u_\varepsilon(x)|^2\,dx = \int_{\mathbb{R}^N}|\nabla U(x)|^2\,dx +O(\varepsilon ^{N-2}),
\end{equation}
and for each $i$,
\begin{equation} \label{eq:3.11}
\int_\Omega\int_\Omega\frac{u_\varepsilon^{2_i^*}(x)u_\varepsilon^{2_i^*}(y)}{|x-y|^{N-\alpha_i}}\,dxdy = \int_{\mathbb{R}^N}\int_{\mathbb{R}^N}\frac{U^{2_i^*}(x)U^{2_i^*}(y)}{|x-y|^{N-\alpha_i}}\,dxdy + O(\varepsilon^{N+\alpha_i}).
\end{equation}
Denote
\[
d_p = \int_{\mathbb{R}^N}U^p(x)\,dx.
\]
Since $2<p<2^*_{min}$, we have $N-(N-2)p<0$ and $N-\frac{(N-2)p}{2}>0$ if $N\geq 4$. Hence,
\[
\begin{split}
\int_{\Omega}u_\varepsilon^p(x)\,dx&\geq \int_{B_r(0)}U_\varepsilon^p(x)\,dx = \bigg(\int_{\mathbb{R}^N}-\int_{B_r^c(0)}\bigg)U_\varepsilon^p(x)\,dx\\
&= \varepsilon^{N-\frac{(N-2)p}2}d_p +O(\varepsilon^{\frac{(N-2)p}2}).\\
\end{split}
\]
and
\[
\begin{split}
\int_{\Omega}u_\varepsilon^p(x)\,dx&\leq \int_{B_{2r}(0)}U_\varepsilon^p(x)\,dx = \bigg(\int_{\mathbb{R}^N}-\int_{B_{2r}^c(0)}\bigg)U_\varepsilon^p(x)\,dx\\
&= \varepsilon^{N-\frac{(N-2)p}2}d_p +O(\varepsilon^{\frac{(N-2)p}2}).\\
\end{split}
\]
Thus,
\begin{equation} \label{eq:3.12}
\int_{\Omega}u_\varepsilon^p(x)\,dx=\varepsilon^{N-\frac{(N-2)p}2}d_p +O(\varepsilon^{\frac{(N-2)p}2}).
\end{equation}

We know that there exists $t_\varepsilon>0$ such that $t_\varepsilon u_\varepsilon\in \mathcal{N}_{\lambda,\Omega}$, that is,
\begin{equation} \label{eq:3.13}
t_\varepsilon^2\int_{\Omega}|\nabla u_\varepsilon|^2\,dx=\lambda t_\varepsilon^{p}\int_{\Omega} u_\varepsilon^p\,dx + \sum_{i=1}^kt_\varepsilon^{22^*_i}\int_\Omega\int_\Omega\frac{u_\varepsilon^{2_i^*}(x)u_\varepsilon^{2_i^*}(y)}{|x-y|^{N-\alpha_i}}\,dxdy.
\end{equation}
By \eqref{eq:3.10} - \eqref{eq:3.12},
\begin{equation} \label{eq:3.14}
\begin{split}
\int_{\mathbb{R}^N}|\nabla U|^2\,dx+O(\varepsilon^{N-2})&=\lambda t_\varepsilon^{p-2}(\varepsilon^{N-\frac{(N-2)p}2}d_p +O(\varepsilon^{\frac{(N-2)p}2})) \\
&+ \sum_{i=1}^kt_\varepsilon^{22^*_i-2}\bigg(\int_{\mathbb{R}^N}\int_{\mathbb{R}^N}\frac{U^{2_i^*}(x)U^{2_i^*}(y)}{|x-y|^{N-\alpha_i}}\,dxdy+ O(\varepsilon^{\frac{N+\alpha_i}2})\bigg).\\
\end{split}
\end{equation}
The fact $U\in \mathcal{M}_{\mathbb{R}^N}$ implies $t_\varepsilon\to 1$ as $\varepsilon \to 0$. Therefore, $t_\varepsilon = 1+o(1)$ and the derivative $t'_\varepsilon$ of $t_\varepsilon$
is given by
\[
t'_\varepsilon=\frac{-\lambda (N-\frac{(N-2)p}2d_p )\varepsilon^{N-\frac{(N-2)p}2-1}(1+o(1))+O(\varepsilon^{\frac{(N-2)p}2-1})}{\sum_{i=1}^k(22^*_i-2)
\bigg(\int_{\mathbb{R}^N}\int_{\mathbb{R}^N}\frac{U^{2_i^*}(x)U^{2_i^*}(y)}{|x-y|^{N-\alpha_i}}\,dxdy+ O(\varepsilon^{\frac{N+\alpha_i}2})\bigg)(1+o(1))+O(\varepsilon^{N-\frac{(N-2)p}2})}.
\]
Hence, there exists $C_p>0$ such that
\[
t'_\varepsilon= -C_p \varepsilon^{N-\frac{(N-2)p}2-1}(1+o(1)).
\]
As a result,
\[
\begin{split}
c_{\lambda,\Omega}&\leq I_{\lambda,\Omega}(t_\varepsilon u_\varepsilon)= \frac{1-2C_p\varepsilon^{N-\frac{(N-2)p}2}(1+o(1))}2\big(\int_{\mathbb{R}^N}|\nabla U(x)|^2\,dx +O(\varepsilon ^{N-2})\big)\\
&-\frac{1- pC_p \varepsilon^{N-\frac{(N-2)p}2}(1+o(1))}{p}\big(\lambda\varepsilon^{N-\frac{(N-2)p}2}d_p +O(\varepsilon^{\frac{(N-2)p}2})\big)\\
&-\sum_{i=1}^k\frac{1- 2_i^*C_p \varepsilon^{N-\frac{(N-2)p}2}(1+o(1))}{2_i^*}\bigg(\int_{\mathbb{R}^N}\int_{\mathbb{R}^N}\frac{U^{2_i^*}(x)U^{2_i^*}(y)}{|x-y|^{N-\alpha_i}}\,dxdy+ O(\varepsilon^{\frac{N+\alpha_i}2})\bigg)\\
&= m(\mathbb{R}^N)- \frac 1p \lambda\varepsilon^{N-\frac{(N-2)p}2}d_p(1+o(1))<m(\mathbb{R}^N)\\
\end{split}
\]
if $\varepsilon>0$ small enough.

\end{proof}

\bigskip

{\bf Proof of Theorem \ref{thm:2}.} Since $m_{\lambda,\Omega} = c_{\lambda,\Omega}$, let $\{u_n\}\subset \mathcal{N}_{\lambda,\Omega}$ be a minimizing sequence of $I_{\lambda,\Omega}$.
Lemmas \ref{lem:3.2} and \ref{lem:3.3} imply $\{u_n\}$ contains a convergent subsequence, the result readily follows.    $\Box$

\bigskip

\bigskip

 \section{Multiple solutions}

\bigskip

In this section, we prove that problem \eqref{eq:1.1} possesses at least $cat_\Omega(\Omega)$ positive solutions. Denote
\[
\Omega_r^+=\{x\in \mathbb{R}^N: d(x,\Omega)\leq r\}\quad{\rm and }\quad\Omega_r^-=\{x\in \Omega: d(x,\partial \Omega)\geq r\}.
\]
Choose $r>0$ so that $\Omega_r^+,\,\Omega^-_r$ and $\Omega$ are homotopically equivalent.

Let
\[
J_\Omega(u) = \frac {1}2 \int_{\Omega}|\nabla u|^2\,dx -\sum_{i=1}^k\frac {1}{22^*_i}\int_{\Omega}\int_{\Omega}\frac{u(x)^{2^*_i}u(y)^{2^*_i}}{|x-y|^{N-\alpha_i}}\,dxdy
\]
and define
\[
\mathcal{M}_{\Omega}=\{u\in H_0^1(\Omega)\setminus\{0\}:\langle J_\Omega'(u), u\rangle =0\}.
\]
 We define as \eqref{eq:3.1} and \eqref{eq:3.4} that
\[
b_{m}= \inf_{\gamma\in\Gamma}\max_{t\in[0,1]}J(\gamma(t)),
\]
where
\[
\Gamma:=\{\gamma\in C([0,1], D^{1,2}(\mathbb{R}^N)): \gamma(0) =0, J(\gamma(1)) \leq 0, \gamma(1)\not = 0\}.
\]
and
\[
b_s=\inf_{u\in D^{1,2}(\mathbb{R}^N)\setminus \{0\}}\sup_{t>0}J(tu).
\]
We also have
\[
m_{\mathbb{R}^N} = b_m=b_s.
\]

We recall that a measurable function $u: \mathbb{R}^N \to \mathbb{R}$ belongs to
the Morrey space $\mathcal{L}^{p,\mu}(\mathbb{R}^N)$ with $p\in [1,\infty)$ and $\mu\in (0,N]$, if and only if
\begin{equation} \label{eq:4.00}
\|u\|^p_{\mathcal{L}^{p,\mu}(\mathbb{R}^N)} = \sup_{R>0,x\in \mathbb{R}^N}R^{\mu-N} \int_{B_R(x)}|u(x)|^p\,dy<\infty.
\end{equation}

\begin{proposition}\label{prop:3.1} Let $\{u_n\}\subset \mathcal{M}_{\Omega}$ be a sequence such that $J_{\Omega}(u_n)\to m(\mathbb{R}^N)$ as $n\to\infty$. Then there exist $(\gamma_n, x_n)\in \mathbb{R}_+\times \Omega$ such that
\[
v_n(x)=\gamma_n^{\frac{N-2}2}u_n(\gamma_n x +x_n)
\]
contains a convergent subsequence in $D^{1,2}(\mathbb{R}^N)$. Moreover, $\gamma_n\to 0$, $x_n\to x\in\bar\Omega$ and $v_n\to U$ in $D^{1,2}(\mathbb{R}^N)$.
\end{proposition}

\begin{proof}
Since $\{u_n\}\subset \mathcal{M}_{\Omega}$, we have
\begin{equation} \label{eq:4.0}
\|u_n\|_{H_0^1(\Omega)}^2= \sum_{i=1}^k\int_{\Omega}\int_{\Omega}\frac{u_n(x)^{2^*_i}u_n(y)^{2^*_i}}{|x-y|^{N-\alpha_i}}\,dxdy.
\end{equation}
Extending $u_n$ to $\mathbb{R}^N$ by setting $u_n=0$ outside $\Omega$ if necessary,   then we have
\begin{equation} \label{eq:4.0a}
m(\mathbb{R}^N)+o(1) = J_\Omega(u_n)= \sum_{i=1}^k(\frac 12-\frac {1}{22^*_i})\int_{\Omega}\int_{\Omega}\frac{u_n(x)^{2^*_i}u_n(y)^{2^*_i}}{|x-y|^{N-\alpha_i}}\,dxdy,
\end{equation}
which implies that for all  $i=1,\cdot\cdot\cdot,k$, there exists $C>0$ such that
\[
\int_{\Omega}\int_{\Omega}\frac{u_n(x)^{2^*_i}u_n(y)^{2^*_i}}{|x-y|^{N-\alpha_i}}\,dxdy\leq C.
\]
By \eqref{eq:4.0}, $\{u_n\}$ is uniformly bounded in $H_0^1(\Omega)$. Moreover, by the definition of $S_i=S_{\alpha_i}$ in \eqref{eq:1.3} and Sobolev embedding, we have
\[
\|u_n\|^2_{H_0^1(\Omega)}=   \sum_{i=1}^k\int_{\Omega}\int_{\Omega}\frac{u_n(x)^{2^*_i}u_n(y)^{2^*_i}}{|x-y|^{N-\alpha_i}}\,dxdy
\leq \sum_{i=1}^kS_i^{\frac{N+\alpha_i}{N-2}}\|u_n\|_{H_0^1(\Omega)}^{\frac{2(N+\alpha_i)}{N-2}}.
\]
Hence, there exists $C>0$ such that $\|u_n\|_{H_0^1(\Omega)}\geq C$, and we deduce from \eqref{eq:4.0} that
\[
\sum_{i=1}^k\int_{\Omega}\int_{\Omega}\frac{u_n(x)^{2^*_i}u_n(y)^{2^*_i}}{|x-y|^{N-\alpha_i}}\,dxdy\geq C.
\]
We may assume that
\begin{equation} \label{eq:4.0b}
\lim_{n\to\infty}\int_{\Omega}\int_{\Omega}\frac{u_n(x)^{2^*_1}u_n(y)^{2^*_1}}{|x-y|^{N-\alpha_1}}\,dxdy\geq \sigma
\end{equation}
for some $\sigma>0$. Extending $u_n$ to $\mathbb{R}^N$ by setting $u_n=0$ outside $\Omega$, we infer from the Hardy-Littlewood-Sobolev inequality that
\begin{equation} \label{eq:4.0c}
\bigg(\int_{\mathbb{R}^N}\int_{\mathbb{R}^N}\frac{u_n(x)^{2^*_1}u_n(y)^{2^*_1}}{|x-y|^{N-\alpha_i}}\,dxdy\bigg)^{\frac{N-2}{2(N+\alpha)}}\leq \|u_n\|_{L^{2^*}(\mathbb{R}^N)}.
\end{equation}
The improved Sobolev inequality
\begin{equation} \label{eq:4.0d}
\|u_n\|_{L^{2^*}(\mathbb{R}^N)}\leq \|u_n\|^\theta_{D^{1,2}(\mathbb{R}^N)}\|u_n\|_{\mathcal{L}^{2,N-2}(\mathbb{R}^N)}^{1-\theta}
\end{equation}
in \cite{PP} together with \eqref{eq:4.0c}  yield
\begin{equation} \label{eq:4.0e}
\bigg(\int_{\mathbb{R}^N}\int_{\mathbb{R}^N}\frac{u_n(x)^{2^*_1}u_n(y)^{2^*_1}}{|x-y|^{N-\alpha_i}}\,dxdy\bigg)^{\frac{N-2}{2(N+\alpha)}}\leq \|u_n\|^\theta_{D^{1,2}(\mathbb{R}^N)}\|u_n\|_{\mathcal{L}^{2,N-2}(\mathbb{R}^N)}^{1-\theta},
\end{equation}
where $\frac {N-2}N\leq \theta<1$. Since $\{u_n\}$ is bounded in $D^{1,2}(\mathbb{R}^N)$, \eqref{eq:4.0b} implies
\[
\|u_n\|_{\mathcal{L}^{2,N-2}(\mathbb{R}^N)}\geq \sigma_1
\]
for some $\sigma_1>0$. Note that for each $n$, the support of $u_n$ is contained in $\Omega$, then there exist $x_n\in \Omega$ and $\gamma_n\in \mathbb{R}$ such that
\begin{equation}\label{eq:4.0f}
\gamma_n^{-2}\int_{B_{\gamma_n}(x_n)}|u_n(y)|^2\,dy\geq \|u_n\|^2_{\mathcal{L}^{2,N-2}}-\frac C{2n}\geq C_1>0.
\end{equation}
Let $v_n(x) = \gamma_n^{\frac{N-2}2}u_n(\gamma_n x +x_n)$. Then
\begin{equation}\label{eq:4.0g}
\int_{B_{1}(0)}|v_n(y)|^2\,dy\geq C_1>0.
\end{equation}
We remark that $\{v_n\}\subset \mathcal{M}_{\mathbb{R}^N}$ and $I(v_n)\to m(\mathbb{R}^N)$ as $n\to \infty$. Therefore, $\{v_n\}$ is bounded in $D^{1,2}(\mathbb{R}^N)$, and we may assume that
\[
v_n\rightharpoonup v_0 \quad {\rm in }\quad D^{1,2}(\mathbb{R}^N);\quad v_n\to v_0 \quad {\in}\quad L^p_{loc}(\mathbb{R}^N);\quad v_n\to v_0\quad a.e.\quad \mathbb{R}^N.
\]
As a result of \eqref{eq:4.0g}, $v_0\not\equiv 0$.

Now, we show that $J(v_0) = b_s$. Indeed, using the fact that $\langle J'(v_n), v_n\rangle = o(1)$, we obtain
\[
\begin{split}
b_s = \sum_{i=1}^k\frac 12\bigg(\frac 12 - \frac 1{2^*_i}\bigg)\int_{\mathbb{R}^N}\int_{\mathbb{R}^N}
\frac{|v_n(x)|^{2^*_i}|v_n(y)|^{2^*_i}}{|x-y|^{N-\alpha_i}}\,dxdy + o(1).
\end{split}
\]
By Lemma \ref{lem:3.1},
\[
\lim_{n\to\infty}\int_{\mathbb{R}^N}\int_{\mathbb{R}^N}
\frac{|v_n(x)|^{2^*_i}|v_n(y)|^{2^*_i}}{|x-y|^{N-\alpha_i}}\,dxdy\geq
\int_{\mathbb{R}^N}\int_{\mathbb{R}^N}
\frac{|v_0(x)|^{2^*_i}|v_0(y)|^{2^*_i}}{|x-y|^{N-\alpha_i}}\,dxdy
\]
and using the fact $\langle J'(v_0), \varphi\rangle = 0$ for $\varphi\in D^{1,2}(\mathbb{R}^N)$, we obtain
\[
b_s \geq \sum_{i=1}^k\frac 12\bigg(\frac 12 - \frac 1{2^*_i}\bigg)\int_{\mathbb{R}^N}\int_{\mathbb{R}^N}
\frac{|v_0(x)|^{2^*_i}|v_0(y)|^{2^*_i}}{|x-y|^{N-\alpha_i}}\,dxdy = J(v_0)\geq m(\mathbb{R}^N)= b_m= b_s.
\]

Next, we show that actually
$v_n$  converges strongly to $v$ in $D^{1,2}(\mathbb{R}^N)$. In fact,  we deduce from $\lim_{n\to\infty}J(v_n) = b_s = J(v_0)$ and $\langle J'(v_n),v_n\rangle = \langle J'(v_0),v_0\rangle + o(1)$ as well as  Lemma \ref{lem:3.1} that
\[
J(v_n-v_0) = o(1)\quad{\rm and}\quad \langle J'(v_n-v_0),v_n-v_0\rangle = o(1)
\]
leading to the result. By Theorem \ref{thm:1}, $v_0=U$.

Finally, we show that $\gamma_n\to 0$ as $n\to\infty$. Suppose on the contrary that $\gamma_n\not\to 0$ as $n\to\infty$, then we have,  up to a subsequence, either $\gamma_n\to\gamma_0>0$ or $\gamma_n\to\infty$.

If $\gamma_n\to\gamma_0$, we see that $supp\, v_n$ is contained in a bounded domain, so does $supp\, v_0$. However, $v_0=U$, which is a contradiction.

If $\gamma_n\to\infty$, the boundedness of $\{u_n\}$ in $H_0^1(\Omega)$ yields
\[
\gamma_n^{-2}\int_{B_{\gamma_n}(x_n)}|u_n(y)|^2\,dy\leq \gamma_n^{-2}\bigg(\int_{\Omega}|u_n(y)|^{2^*}\,dy\bigg)^{\frac{N-2}N}\bigg(\int_\Omega1\,dy\bigg)^{\frac 2N}\leq \frac C{\gamma_n}\to 0
\]
as $n\to\infty$, which contradicts to \eqref{eq:4.0f}.

\end{proof}

\bigskip
We may assume that $B_r(0)\subset\Omega$ and define
\begin{equation} \label{eq:4.1}
m(\lambda, r) = m_{\lambda, B_r(0)}.
\end{equation}
By Theorem \ref{thm:2},  $m(\lambda, r)$ is achieved by a positive function, which is radially symmetric about the origin.
We remark  that $m(\lambda, r)$ does not depend on the choice of the center of the ball, but only on the radius. So for every $x\in \mathbb{R}^N$, we have $m(\lambda, r) = m(\lambda, B_r(x))$. Obviously, $m(\lambda, \Omega)<m(\lambda, r)$.

Denote
\[
\|u\|^{22^*_i}_{HL}=\int_{\mathbb{R}^N}\int_{\mathbb{R}^N}\frac{u^{2_i^*}(x)u^{2_i^*}(y)}{|x-y|^{N-\alpha_i}}\,dxdy.
\]
Let $\eta\in C_c^\infty(\mathbb{R}^N, \mathbb{R}^N)$ be such that $\eta(x)=x$ for all $x\in\bar\Omega$.
We introduce the barycenter of a function $u\in H^1(\mathbb{R}^N)$ as
\begin{equation} \label{eq:4.2}
\beta(u) = \sum_{i=1}^k\frac{1}{\|u\|^{22^*_i}_{HL}}\int_{\mathbb{R}^N}\int_{\mathbb{R}^N}\frac{\eta(x)u^{2_i^*}(x)u^{2_i^*}(y)}{|x-y|^{N-\alpha_i}}\,dxdy.
\end{equation}

\begin{lemma}\label{lem:4.2} There exists $\lambda^*>0$ such that if $\lambda\in(0,\lambda^*)$ and $u\in \mathcal{N}_{\lambda,\Omega}$ satisfying $I_{\lambda,\Omega}(u)<m(\lambda,r)$, then
$\beta(u)\in \Omega_r^+$.
\end{lemma}

\begin{proof} We argue indirectly. Suppose by condition that there exist  $\lambda_n\to0$,  $u_n\in \mathcal{N}_{\lambda_n,\Omega}$ such that
\begin{equation} \label{eq:4.3}
m(\lambda_n,\Omega)\leq I_{\lambda_n,\Omega}(u_n)\leq m(\lambda_n,r)<m(\mathbb{R}^N),
\end{equation}
but $\beta(u_n)\not\in \Omega_r^+$. Define
\[
g_{u_n}(t) = I_{\lambda_n,\Omega}(tu_n)= \frac {t^2}2 \int_{\Omega}|\nabla u_n|^2\,dx-\frac {t^p}p\lambda_n\int_{\Omega} u_n^p\,dx -\sum_{i=1}^k\frac {t^{22^*_i}}{22^*_i}\int_{\Omega}\int_{\Omega}\frac{u_n(x)^{2^*_i}u_n(y)^{2^*_i}}{|x-y|^{N-\alpha_i}}\,dxdy.
\]
For each fixed $n$, $g_{u_n}(0)=0,\, g_{u_n}(t)\to -\infty$ if $t\to\infty.$ Thus there exists a maximum point $t_n>0$ of  $g_{u_n}(t)$ such that $g_{u_n}'(t_n)=0$.
We claim that the maximum point $t_n>0$ is unique. Indeed,
\[
g_{u_n}'(t)=t \int_{\Omega}|\nabla u_n|^2\,dx- t^{p-1}\lambda_n\int_{\Omega} u_n^p\,dx -\sum_{i=1}^kt^{22^*_i-1}\int_{\Omega}\int_{\Omega}\frac{u_n(x)^{2^*_i}u_n(y)^{2^*_i}}{|x-y|^{N-\alpha_i}}\,dxdy=th_{u_n}(t),
\]
where
\[
h_{u_n}(t)=\int_{\Omega}|\nabla u_n|^2\,dx- t^{p-2}\lambda_n\int_{\Omega} u_n^p\,dx -\sum_{i=1}^kt^{22^*_i-2}\int_{\Omega}\int_{\Omega}\frac{u_n(x)^{2^*_i}u_n(y)^{2^*_i}}{|x-y|^{N-\alpha_i}}\,dxdy.
\]
Apparently, there exists unique $t_n>0$ such that $h_{u_n}(t_n)=0$, which is the maximum point of $g_{u_n}(t)$. Since $u_n\in \mathcal{N}_{\lambda_n,\Omega}$, we
get $t_n=1$, that is,
\begin{equation} \label{eq:4.4}
I_{\lambda_n,\Omega}(u_n)=\sup_{t\geq 0}I_{\lambda_n,\Omega}(t u_n).
\end{equation}
Moreover, by the definition of $S_i=S_{\alpha_i}$ in \eqref{eq:1.3} and Sobolev embedding, we may prove as the proof of Proposition \ref{prop:3.1} that there exists $C>0$ such that
$C^{-1}\leq\|u_n\|_{H_0^1(\Omega)}\leq C$.
Solving
$$
\|su_n\|_{H_0^1(\Omega)}^2= \sum_{i=1}^ks^{22^*_i}\int_{\Omega}\int_{\Omega}\frac{u_n(x)^{2^*_i}u_n(y)^{2^*_i}}{|x-y|^{N-\alpha_i}}\,dxdy,
$$
we find there exists a unique $s_n>0$ such that $s_nu_n\in \mathcal{M}_{\lambda_n,\Omega}$, that is,
$$
\|u_n\|_{H_0^1(\Omega)}^2= \sum_{i=1}^ks_n^{22^*_i-2}\int_{\Omega}\int_{\Omega}\frac{u_n(x)^{2^*_i}u_n(y)^{2^*_i}}{|x-y|^{N-\alpha_i}}\,dxdy.
$$
The boundedness of $\{u_n\}$ in $H_0^1(\Omega)$ implies that $\{s_n\}$ is bounded.

Let $v_n(x)=s_nu_n(x)$. Then, $\{v_n\}$ is uniformly bounded in $H_0^1(\Omega)$ and $\lambda_n\int_\Omega |v_n|^p\,dx\to 0$ as $n\to\infty$.
We conclude by \eqref{eq:4.4} that
\begin{equation} \label{eq:4.5}
\begin{split}
m(\mathbb{R}^N)&\leq J(v_n)=J(s_nu_n) = I_{\lambda_n,\Omega}(s_nu_n) + \lambda_n\int_\Omega |v_n|^p\,dx\\
&\leq I_{\lambda_n,\Omega}(u_n)+o(1)\leq m(\mathbb{R}^N)+o(1),\\
\end{split}
\end{equation}
which yields
\[
J(v_n)\to m(\mathbb{R}^N)
\]
as $n\to\infty$. By Proposition \ref{prop:3.1}, there exist $\gamma_n\in \mathbb{R}$ and $x_n\in \Omega$ such that $x_n\to x_0\in \bar\Omega$, $w_n(x) =\gamma_n ^{\frac {N-2}2}v_n(\gamma_n x+x_n)\to U(x)$ in $D^{1,2}(\mathbb{R}^N)$ as $n\to\infty$. Therefore, the Lebesgue dominated theorem yields
\begin{equation} \label{eq:4.6}
\begin{split}
\beta(u_n)&=\beta(s_nu_n) = \beta(v_n)= \sum_{i=1}^k\frac{1}{\|v_n\|^{22^*_i}_{HL,}}\int_{\mathbb{R}^N}\int_{\mathbb{R}^N}\frac{\eta(x)v_n^{2_i^*}(x)v_n^{2_i^*}(y)}{|x-y|^{N-\alpha_i}}\,dxdy\\
&= \sum_{i=1}^k\frac{1}{\|w_n\|^{22^*_i}_{HL,}}\int_{\mathbb{R}^N}\int_{\mathbb{R}^N}\frac{\eta(\gamma x+x_n)w_n^{2_i^*}(x)w_n^{2_i^*}(y)}{|x-y|^{N-\alpha_i}}\,dxdy,\\
&\to x_0\in\bar\Omega.
\end{split}
\end{equation}
This contradicts the assumption $\beta(u_n)\not\in \Omega_r^+$. The assertion follows.

\end{proof}

\bigskip

{\bf Proof of Theorem \ref{thm:3}:} Let
$I_\lambda^c(\mathcal{N}_{\lambda,\Omega})=\{u\in \mathcal{N}_{\lambda,\Omega}: I_\lambda(u)\leq c\}$ be a level set. We claim that
\[
cat_{I_{\lambda,\Omega}^{m(\lambda,r)+\lambda^{q*} \delta}(\mathcal{N}_{\lambda,\Omega})}(I_{\lambda,\Omega}^{m(\lambda,r)+\lambda^{q*} \delta}(\mathcal{N}_{\lambda,\Omega}))\geq cat_\Omega(\Omega).
\]
Indeed, let us define $\gamma:\Omega_r^-\to I_{\lambda,\Omega}^{m(\lambda,r)+\lambda^{q*} \delta}(\mathcal{N}_{\lambda,\Omega})$ by
\[
\gamma(y)(x) = v(x-y) \quad {\rm if }\quad x\in B_r(y);\quad \gamma(y)(x) = 0 \quad {\rm if }\quad x\not\in B_r(y),
\]
where $v\in H^1_0(B_r(0))$  is a positive minimizer of $m(\lambda,r)$. We can assume that $v$ is radial, see \cite{MS} and references therein. Therefore, $\beta\circ \gamma = id :\Omega_r^-\to\Omega_r^-$. Suppose now that
\[
n= cat_{I_{\lambda,\Omega}^{m(\lambda,r)+\lambda^{q*} \delta}(\mathcal{N}_{\lambda,\Omega})}(I_{\lambda,\Omega}^{m(\lambda,r)+\lambda^{q*} \delta}(\mathcal{N}_{\lambda,\Omega})).
\]
Then
\[
I_{\lambda,\Omega}^{m(\lambda,r)+\lambda^{q*} \delta}(\mathcal{N}_{\lambda,\Omega})=\cup_{i=1}^nA_i,
\]
where $A_i,i=1,...,n,$ is closed and contractible in $I_{\lambda,\Omega}^{m(\lambda,r)+\lambda^{q*} \delta}(\mathcal{N}_{\lambda,\Omega})$, that is, there exists $h_i\in C([0,1]\times A_i, I_{\lambda,\Omega}^{m(\lambda,r)+\lambda^{q*} \delta}(\mathcal{N}_{\lambda,\Omega}))$ such that, for every $u,v\in A_i$,
\[
h_i(0,u)=u,\quad h_j(1,u) =  h_j(1,v).
\]
Let $B_i = \gamma^{-1}(A_i), 1\leq i\leq n$.  Then $B_i$ are closed and
\[
\Omega_r^- =\cup_{i=1}^nB_i.
\]
By  Lemma \ref{lem:4.2} and the deformation
\[
g_i(t,x)=\beta(h_i(t,\gamma(x))),
\]
we see that $B_i$ is contractible in $\Omega_r^+$. It follows that
\[
cat_\Omega(\Omega)= cat_{\Omega_r^+}(\Omega_r^-)\leq \sum_{i=1}^ncat_{\Omega_r^+}(B_i)=n.
\]
Hence, $I_{\lambda,\Omega}$ has at least $cat_\Omega(\Omega)$ critical points on $\mathcal{N}_{\lambda,\Omega}$. The proof is completed. $\Box$

\bigskip

\bigskip

\bigskip

\bigskip

\vspace{2mm} \noindent{\bf Acknowledgment.}  This work was supported by NNSF of China (No: 12171212 and No: 11771300).

\end{document}